\documentclass[12pt,reqno]{amsart}
\usepackage[french, english]{babel}
\usepackage[normalem]{ ulem }
\usepackage{soul}
\usepackage{eucal}
\usepackage{flafter}
\usepackage{amscd,amsthm}
\usepackage{amsfonts}
\usepackage{graphicx}
\usepackage{hyperref}
\usepackage{cancel}

\numberwithin{equation}{section}
\usepackage{amsmath}
\usepackage{amssymb,amsthm}
\usepackage{amsbsy,amsfonts,amscd}
\usepackage{mathtools} 
\mathtoolsset{showonlyrefs}
\usepackage{epsfig}
\usepackage{psfrag}
\usepackage{color,verbatim}
\usepackage{graphics,graphicx}
\usepackage{ulem}
\usepackage[charter]{mathdesign}

\usepackage{xcolor}
\usepackage{pstricks,pst-node}
\usepackage{tikz}
\usetikzlibrary{backgrounds}
\usetikzlibrary{patterns,fadings}
\usetikzlibrary{arrows,decorations.pathmorphing}
\usetikzlibrary{shapes}
\usetikzlibrary{calc}

\usepackage{enumerate}

\definecolor{brightcerulean}{rgb}{0.11, 0.67, 0.84}
\definecolor{cerulean}{rgb}{0.0, 0.48, 0.65}
\definecolor{Gray}{rgb}{0.5, 0.5, 0.5}

\usepackage{bm}
\newcommand*{\BF}[1]{\ifmmode\bm{#1}\else\textbf{#1}\fi}

\newcommand\bx{{\mathbf x}}
\newcommand\by{{\mathbf y}}

\newcommand\bq{{\mathbf q}}

\newcommand\EE{{\mathbb E}}
\newcommand\PP{{\mathbb P}}

\newcommand\RR{{\mathbb R}}

\newcommand\R{{\mathbb R}}

\newtheorem{Theorem}{Theorem}

\newtheorem{Definition}{Definition}

\newtheorem{Remark}{Remark}

\newcommand{\mc}[1]{{\mathcal #1}}

\newcommand{\bb}[1]{{\mathbb #1}}

\newcommand{\ve}{\varepsilon}

\begin{document}

\title[]{Gamma Convergence approach for the large deviations of the density in systems of interacting diffusion processes }

\author{J. Barr\'e}

\address{J. Barr\'e, Institut Denis Poisson, Universit\'e d'Orl\'eans, CNRS et Universit\'e de Tours, et Institut Universitaire de France}
\email{{\tt julien.barre@univ-orleans.fr}}

\author{C.Bernardin}
\address{C. Bernardin, Universit\'e C\^ote d'Azur, CNRS, LJAD\\
Parc Valrose\\
06108 NICE Cedex 02, France}
\email{{\tt cbernard@unice.fr}}

\author{R. Ch\'etrite}

\address{R. C\'etrite.Universit\'e C\^ote d'Azur, CNRS, LJAD\\
Parc Valrose\\
06108 NICE Cedex 02, France}
\email{{\tt raphael.chetrite@unice.fr}}

\author{Y. Chopra}

\address{Y. Chopra, Universit\'e C\^ote d'Azur, CNRS, LJAD\\
Parc Valrose\\
06108 NICE Cedex 02, France}
\email{{\tt yash.chopra@unice.fr}}

\author{M. Mariani}
\address{M. Mariani, Faculty of Mathematics\\
National Research University Higher School of Economics \\
6 Usacheva St., 119048 Moscow, Russia.}
\email{{\tt mmariani@hse.ru}}

\date{\today.}
\begin{abstract} We consider extended slow-fast systems of $N$ interacting diffusions. The typical behavior of the empirical density is described by a nonlinear McKean-Vlasov equation depending on $\varepsilon$, the scaling parameter separating the time scale of the slow variable from the time scale of the fast variable. Its atypical behavior is encapsulated in a large $N$ Large Deviation Principle (LDP) with a rate functional ${\mathcal I}^\ve$. We study the $\Gamma$-convergence of $\mathcal I^\ve$ as $\ve \to 0$ and show it converges to the rate functional appearing in the Macroscopic Fluctuations Theory (MFT) for diffusive systems. 

\end{abstract}



\thanks{ }
\newtheorem{theorem}{Theorem}[section]
\newtheorem{proposition}[theorem]{Proposition}
\newtheorem{lemma}[theorem]{Lemma}
\newtheorem{corollary}[theorem]{Corollary}
\newtheorem{assumption}[theorem]{Assumption}
\theoremstyle{remark}
\newtheorem{remark}[theorem]{Remark}
\newtheorem{definition}[theorem]{Definition}
\makeatletter
\newcommand{\cqfd}{{\unskip\kern 6pt\penalty 500
\raise -2pt\hbox{\vrule\vbox to 6pt{\hrule width 6pt
\vfill\hrule}\vrule}\par}}

\maketitle
\tableofcontents
\section{Introduction}
\label{intro}

We consider a system of $N\ge 1$ interacting particles (e.g. economical agents, living or artificial entities ..). The configuration of a particle labeled by $i$  is described by two coordinates: a first one (called position for convenience) $q_i  \in \RR^n$ and a second one (called internal degree of freedom) $\theta_i$ living in some $m$-dimensional Riemannian manifold $(\mc M, {\mathfrak g})$ whose Riemannian measure is denoted by $\mu_{\mathfrak g}$. The gradient {\footnote{In local coordinates, with Einstein's convention, for any smooth function $f$ and any vector field $X:=X^k \partial_{\theta_k}$, $\nabla_\theta f= {\mathfrak g}^{k \ell} \partial_{\theta_k} f \partial_{\theta_\ell}$ and $\nabla_\theta \cdot [X^k \partial_{\theta_k}] =\tfrac{1}{\sqrt G} \partial_{\theta_k} (X^k \sqrt{G})$ where $G={\rm{det}}({\mathfrak g}^{k\ell})$. We have also then the integration by parts formula: $\int d\mu_{\mathfrak  g} \, X (\nabla_\theta f) = - \int d\mu_{\mathfrak g}\, (\nabla_\theta \cdot X) \; f$.  }} on ${\mathcal M}$ is denoted by $\nabla_\theta$ and the divergence by $\nabla_\theta \cdot$ . The equations of motion are given by Fisk-Stratonovich stochastic differential equations (SDE's):\\
\begin{equation}
\label{eq:general-model-manifold0}
\begin{cases}
& dq_i = \varepsilon \, V (\theta_i) dt,\\
& {d\theta}_i =\Big[ B  - \tfrac{1}{\mathcal{N}_i}\sum_{j\in \mathcal{V}_i} F (\cdot , \theta_j) \Big]\,  (\theta_i)  \, dt \;  + \; \sqrt{2} \, \sum_{a=1}^\ell A_a (\theta_i) \, \circ dW_i^a (t).
\end{cases}
\end{equation}
Here, $V:=V(\theta)$ is a vector field on $\R^n$; $B, A_1, \ldots, A_\ell$ are $\ell +1 $ vector fields on ${\mc M}$ ($\ell$ is arbitrary); and for each $\theta' \in {\mc M}$,  $F:=F(\cdot, \theta')$ is a vector field on ${\mc M}$ deriving from a  potential $W(\cdot, \theta')$: 
$$F(\theta, \theta')=\nabla_\theta W (\theta, \theta').$$
All these fields are assumed to be smooth. The set $\mathcal{V}_i$ is the set of labels of particles interacting with particle $i$ in a neighborhood of radius $R>0$:
$$\mc V_i :=\{ j\in \{1, \ldots, N\} \; ;\; \mid q_i - q_j \mid \; \le\;  R\}$$
and $\mathcal{N}_i$ is the number of particles in $\mathcal{V}_i$. The $W_i:=(W_i^1, \ldots, W_i^m)$'s are independent standard $m$-dimensional Wiener processes simulating the interaction with some external  environment.

This class of models includes several types of active matter models (see for instance \cite{Peruani08,Baskaran08,Degond08,Cattiaux,Degond19}) born after the seminal work of Vicsek et al.\cite{Vicsek}; in these models ${\mathcal M}$ is often $\mathbb{S}^1$, but may be $\mathbb{S}^2$ or $SO_3$. Note however that \eqref{eq:general-model-manifold0} is sufficiently general to have applications in other fields (for example as simplified Lagrangian stochastic model \cite{Bossy}). A particular case of interest in active matter \cite{Peruani08} is the two dimensional ($n=2$) model with ${\mc M}=\{e^{i \alpha}\; ; \; \alpha \in [-\pi, \pi)\}$ the unit circle equipped with the trivial metric and
\begin{equation*}
\label{eq:pot-pc}
\begin{split}
& V (e^{i\alpha}) =e^{i\alpha} \in \R^2, \quad W (e^{i\alpha}, e^{i \alpha'}) =\cos (\alpha -\alpha'), \\
& A_1 (e^{i\alpha}) =1,  \quad \ell=1.
\end{split}
\end{equation*}
A natural multidimensional generalization of this model follows by the choice ${\mathcal M}={\mathbb S}^m$ the $m$-dimensional sphere equipped with its natural metric and 
\[
W (\theta, \theta') :=- \, \theta \cdot \theta', \quad V(\theta) :=\theta,
\]
where $B, A_a$ are arbitrary and $\cdot$ denotes the usual scalar product in $\R^{m+1}$. Hence here the velocity  $\theta_i$ of the particle $i$ has a constant norm by hypothesis. \\

In this work we will consider large systems, i.e.  $N\to \infty$, as $\ve \to 0$, i.e. assuming that the $q_i$'s dynamics is much slower than the $\theta_i$'s one. Hence our model belongs to the class of \textit{infinite dimensional} slow-fast systems. \\

A huge amount of work has been devoted to the study of {\it{finite}} {\it{dimensional}} (random or deterministic) slow-fast dynamical systems, of which \eqref{eq:general-model-manifold} is only a particular subclass. Hence $N$ is fixed and $\ve \to 0$, i.e. $N\ve \to 0$. For these finite-dimensional models, one is interested in the  characterization of the dynamics of the slow variables $q(t) \in (\R^{n})^N$  as $\ve \to 0$. Its typical behavior, in the time scale $\ve^{-1}$, is studied by tools of homogeneization theory \cite{BenS, Papa, Bakhvalov, Jikov, Cioranescu, Olla,  Landim, Pavliotis}. Since the initial system is random, fluctuations of $q (\ve^{-1} t)$ around its typical behavior $\bar q (t)$ are also of interest and can be studied theoretically. In particular LDP  {\footnote{See \cite{Varadhan,Varadhan1,Varadhan2, FW, Ellis, Deuschel, DemboZ,DenH,Touchette} for a general introduction about LDP.}} exist in the form \cite{FW,Veretennikov, Lipster, Bouchet0}
\begin{equation}
\label{eq:LDP1-intro}
\PP \left( q( \ve^{-1}t ) \approx Q(t) \text{ on } [0,T] \right) \sim \exp \left(-\ve^{-1} {\mathcal J}^N_T (Q) \right)
\end{equation}
where ${\mathcal J}^N_T$ is an explicit rate functional vanishing for $Q={\bar q}$.\\

On the other hand, for fixed $\varepsilon$, one can be interested in the description of the dynamics (in $q$ and $\theta$) as $N\to \infty$, i.e. $N\ve \to \infty$, through the study of the empirical density $f_N^{\ve} (q, \theta, t)$. The dynamics becomes thus {\textit{infinite-dimensional}} and the typical behavior of $f_N^\ve (q, \theta,t)$ is described by $f^\ve (q, \theta,t)$ which is solution of a (kind of) McKean-Vlasov equation \cite{McKean, McKean2, Desai, Kipnis,  Dawson0, Funaki, Oel,  Bonilla, Gartner, Sznitman, Meleard, Carillo-Delgadino-Pavliotis}. Fluctuations (central limit theorems or large deviations principles) around this typical behavior have been investigated previously \cite{Tanaka,Dawson3,Kipnis2, Dawson5, Dawson4,Dai Pra, Fontbona,  Budhiraja, Fischer,Bar1,Bouchet, Muller, Reygner, Dos,Coppini}. More explicitly a large deviations principle for $f_N^\ve$ holds{\footnote{Sometimes it is also necessary to perform first a change of frame, see \eqref{eq:changetsfeps}}}:
\begin{equation}
\label{eq:LDP2-intro}
\PP \left( f_N^\ve (q, \theta, t\ve^{-2}) \approx g(q, \theta, t) \text{ on } [0,T]  \right) \sim \exp ( -N {\mathcal I}_T^\ve (g) ) 
\end{equation}
where the rate functional $\mc I_T^\ve$ is of course vanishing if $g(\cdot, t) \equiv f^\ve (\cdot, t \ve^{-2})$ on the time interval $[0,T]$.\\

In this paper we are interested in the behavior of the large deviations functional ${\mathcal I}_T^\ve$ for the empirical density {\footnote{While the interaction is mean field we will send $R\to 0$ after $N\to \infty$ so that the binary interaction will become local in space, but this is not a fundamental  aspect of our work, even if the results would have to be modified.}}  when $\ve \to 0$. From a technical point of view the study of this convergence of functionals has to be accomplished in the $\Gamma$-convergence framework \cite{Braides,DelMaso}. Roughly speaking we show, under a certain number of assumptions on the model, that ${\mathcal I}^\ve_T$ converges as $\ve \to 0$ to a functional ${\mathcal I}_T$ whose finite values are supported on density functions $g$ which have a local equilibrium form: $g(q, \theta, t)=\rho (q, t) G(\theta)$ where $G(\theta)$ is the unique stationary measure --  in the fast dynamics variables $\theta$ -- of the McKean-Vlasov equation (i.e. when $\ve =0$), while $\rho (q,t)$ is arbitrary and describes the potential time dependent density profiles (in $q$) available by the slow dynamics of the $q_i$'s. Hence, in some sense, we establish some averaging (or homogeneization) principle at the level of large deviations. The limiting large deviations functional ${\mc I}_T$ takes a form similar to the one appearing in the context of the Macroscopic Fluctuations Theory  \cite{Bertini0, Bertini} for diffusive systems, and is fully explicit. In particular, the functional ${\mc I}_T$ vanishes when $g(q, \theta, t) =\rho (q, t) G(\theta)$ where $\rho$ is the solution of a linear diffusion equation which can also be guessed by a Chapman-Enskog expansion \cite{Chapman} of the solution $f^\ve$  of the McKean-Vlasov equation mentioned above. Our limiting large deviation functional ${\mathcal I}_T$ is also consistent with a Chapman-Enskog analysis of the so-called ``Dean equation" (fluctuating McKean-Vlasov equation at finite $N$). We point out that the active matter systems, which are one of the motivations of this work, usually feature a moderately large number of individual units (typically much smaller than for a standard fluid for instance); a precise description of the finite $N$ fluctuations, as provided here at the large deviation level, may then be particularly important. 
The main limitation of our work is the crucial assumption that the equilibrium state $G$ is unique while in many cases of interest (and in particular in active matter models) it is not true. A very interesting question is therefore to know how to extend our results in these cases.

\subsection{Plan}
The paper is organized as follows. In Section \ref{sec:kl} we present the model and describe its kinetic limit, as well as its approximated hydrodynamics when the spatial dynamics is much slower than the angular dynamics, by relating it to the classical Chapman-Enskog approach. We then introduce the finite size fluctuations kinetic equation that we reinterpret in the large deviation (LD) theory framework. Our first main result is then stated in Section \ref{sec:mainresult} and establishes a LD principle with an explicit rate function for the density of particles in the limit where the spatial dynamics is much slower than the angular dynamics. Since the limit involves convergence of rate functionals we have to use the appropriate notion of $\Gamma$-convergence. The proof of this result is given in Section \ref{sec:proofgamma}. The paper is concluded by several appendices.

\section{From the microscopic model to a fluctuating hydrodynamic equation}
\label{sec:kl}

\subsection{Microscopic models}

While our main result (Theorem \ref{thm:main}) could probably be extended for the model given by \eqref{eq:general-model-manifold0} under some assumptions on the vector fields $V$, $B$, $F$ and $A_a$'s, we choose for technical reasons (in particular ones leading to Appendix \ref{sec:app-le}  and Appendix \ref{sec:linearizedoperator} where our `dissipative assumption' \eqref{eq:vraiehypothese} can be checked) to focus only on `$A_a$'s-gradient dynamics ', i.e.  
\begin{equation}
\label{eq:general-model-manifold}
\begin{cases}
& dq_i = \varepsilon \, V (\theta_i) dt,\\
& {d\theta}_i = - \sum_{a=1}^m \left[ {\mathfrak g}  \Big( A_a\, , \, \nabla_\theta {U}+ \tfrac{1}{\mathcal{N}_i}\sum_{j\in \mathcal{V}_i} F (\cdot , \theta_j) \Big)  \; A_a \right]\,  (\theta_i)  \, dt \\
&\quad  \quad \; + \; \sum_{a=1}^m \Big[ (\nabla_\theta \cdot  A_a)\,  A_a \Big]\,  (\theta_i) \, dt \;  + \; \sqrt{2} \, \sum_{a=1}^m A_a (\theta_i) \, \circ dW_i^a (t), 
\end{cases}
\end{equation}
where we recall that ${\mathfrak g}$ is the Riemannian metric on ${\mc M}$.  We also assume that $\mc M$ is compact. The presence of the spurious drift term $\sum_{a=1}^m [\nabla_\theta \cdot  A_a] \, A_a $ is here to ensure that the dynamics of the $\theta_i$'s is reversible {\footnote{This reversibility means that if $L$ is the Markovian generator with $q$
frozen acting on function $f$ on $\mathcal{M}$ as 
\[
L(f)=\sum_{a=1}^m  e^{\mathcal U} \nabla_{\theta}.\left(
e^{-\mathcal{U}} \, {\mathfrak g} \Big(\nabla_{\theta} f , A_{a} \Big) \, A_a\right), 
\]
then for any function $f,h$ on $\mathcal{M}$ the integral $\int_{\mathcal{M}}d \mu_{\mathfrak g}\; e^{-\mathcal{U}} \, f\, Lh$
is symmetric in $f,h$.}} 
with respect to the Gibbs measure  $e^{-\mathcal{U}}$, $
\mathcal{U}\left(\theta\right)\equiv\sum_{i}\left(U\left(\theta \right)+\frac{1}{2\mathcal{N}_{i}}\sum_{j\in\mathcal{V}_{i}}W\left(\theta ,\theta_{j}\right)\right)
$
when $R=\infty$, and the potential $W$ is symmetric, i.e. $W(\theta, \theta')=W(\theta', \theta)$. The interaction is thus regulated by $A(\theta):= (A_1 (\theta), \ldots, A_m (\theta))$ that we assume to satisfy: for any smooth function $f(\theta)$ on ${\mathcal M}$,
\begin{equation*}
\sum_{a=1}^m \int_{\mathcal M} d\mu_{\mathfrak g} (\theta) \, (A_a f)^2 (\theta) =0 \quad \text{implies} \quad f\equiv 0.
\end{equation*}
This condition is here to ensure a non-degenerate diffusivity in the $\theta$ variable. We also assume a non-degeneracy condition for $V$:
\begin{equation}
\label{eq:nondegenerateV}
{\rm{Span}} \left\{ \nabla_\theta V (\theta)\; ;\; \theta \in {\mc M} \right\} =\R^n.\\
\end{equation}
\vspace{0.5cm}

For the convenience of the reader we will write explicitly the proof for $\mc M:=(-\pi, \pi]$ the unit torus equipped with the trivial metric but we will state all our results in the general case presented above. The interested reader will check easily that our proofs can be extended mutatis mutandis to the models described by \eqref{eq:general-model-manifold}. In this simpler case, the equations of motion \eqref{eq:general-model-manifold} are thus given by the  Fisk-Stratonovich SDE's (with $m=1$ and by defining $A_{1} (\theta)= \sqrt{\Gamma (\theta)} \nabla_\theta$) which can be translated as the Ito SDE's:\\
\begin{equation}
\label{eq: motion}
\begin{split}
&dq_i  = \varepsilon \, V (\theta_i) dt,\\
&{d\theta}_i = - [\Gamma \partial_\theta {\bb U}] (\theta_i)dt - \frac{1}{\mathcal{N}_i}\sum_{j\in \mathcal{V}_i} \Gamma (\theta_i) F (\theta_i, \theta_j)dt  +\sqrt{2 \Gamma(\theta_i)} \, dW_i(t)
\end{split}
\end{equation}
with the effective potential
\begin{equation*}
{\bb U}(\theta):=U(\theta) - \log \Gamma(\theta).
\end{equation*}
Since $\mc M:=(-\pi, \pi]$ is the unit torus all these fields can be seen as $2\pi$-periodic functions in the internal degree of freedom variable.

\subsection{Thermodynamic limit}
\subsubsection{Kinetic equation}
\label{susubsecKE}

Let us first fix $\ve>0$. In the thermodynamic limit $N\to \infty$ and then local spatial limit $R\to 0$, at the kinetic level, the time dependent density $f^\ve(q,\theta,t)$ of the system is described by a kinetic equation (see Appendix \ref{app:kin} for a formal derivation and \cite{Bossy} for a rigorous derivation in a similar context) which is a kind of Mc-Kean-Vlasov equation. More exactly it is an integro (in $\theta$)-differential (in $q-\theta$) non-linear Fokker-Planck equation  \cite{McKean, McKean2, Desai, Kipnis,  Dawson0, Funaki, Oel,  Bonilla, Gartner, Sznitman, Meleard, Carillo-Delgadino-Pavliotis}:
\begin{equation}
\label{eq:det-fluct-eq}
\begin{split}
\partial_t f^\ve &=\partial_\theta \left( \Gamma \left[ \partial_\theta U  +\cfrac{ F (f^\ve) }{\rho^\ve}\right]   f^\ve  +  \Gamma \, \partial_\theta f^\ve \right) -\; \varepsilon V \cdot \nabla  f^\ve \\
&:= {\mc D}_{f^\ve} (f^\ve) - \ve \mc T (f^\ve)
\end{split}
\end{equation}
with $F(f)$ meaning 
\begin{equation*}
F(f) (q, \theta):= \int_{-\pi}^\pi d\theta' \, F (\theta, \theta') \, f (q, \theta') \, d\theta', 
\end{equation*}
and 
\begin{equation*}
\rho^\ve (q):=\Pi (f^\ve) (q):= \int_{-\pi}^{\pi} f^\ve (q, \theta')  d\theta'.
\end{equation*}
Here the linear dissipative operator ${\mc D}_f$ and the linear transport operator $\mc T$ are defined for all function $g$ by 
\begin{align}
\label{eq:defDT}
&{\mc D}_f (g) := \partial_\theta \left(\Gamma  \left[ \partial_\theta U +\,  \cfrac{ F (g) }{\Pi (g) }\right]  g  +  \Gamma \, \partial_\theta g \right) ,\\
&\mc T (g) := V \cdot \nabla g.
\end{align}

\subsubsection{Local equilibiria}
\label{subsubsec:localEquilibriA}
The fast dynamics ($\ve=0$) is given by
\begin{equation}
\partial_t f =\mathcal{D}_f (f).
\label{eq:fastdyn}
\end{equation}
The time asymptotic stationary solutions $f_{\rm{le}}$ of  \eqref{eq:fastdyn} are called local equilibria. These local equilibria are studied in Appendix \ref{sec:app-le} where it is shown that they take the form $f_{\rm{le}} (q, \theta) = \rho (q) G_{\rho (q)} (\theta)$ where 
$$\rho (q) := \int_{-\pi}^\pi d\theta f_{\rm{le}} (q, \theta)$$
and $G:=G_\rho$ is a solution of
\begin{equation}
\label{eq:fluxnul}
[\partial_\theta U +  F(G)]\,  G + \partial_\theta G =0.
\end{equation}
with the condition 
\begin{equation*}
\int_{-\pi}^\pi d\theta G(\theta) =1.
\end{equation*}
In the sequel we restrict our study to the case where we have only one solution to this equation that we denote by $G$. Then all local equilibrium $f_{\rm{le}}$ is in the form 
\begin{equation}
f_{\rm{le}}(q,\theta) = \rho(q) \; G(\theta)
\label{eq:le-generic}
\end{equation}
where $G>0$ is unique and fixed and $\rho\ge 0$ is arbitrary. For generic potentials $U$ and $W$, it is difficult to precise exactly under which conditions this occurs. However, as shown in Appendix  \ref{sec:app-le}, if the interaction potential $W$ is sufficiently small, this is the case.  A detailed study of the the set of local equilibria for related McKean-Vlasov models can be found for example in \cite{Desai,  Dawson0, Bonilla,Chayes,Tug14, Carillo2,Degond13}.\\

In the following, the expectation of $f$ with respect to $G$ is written $\langle f \rangle_G$ and the corresponding scalar product between functions $f$ and $g$ by $\langle f , g \rangle_G =\int_{-\pi}^\pi f g  G(\theta) d \theta$.\\

\subsubsection{The hydrodynamic limit via Chapman-Enskog expansion: Transport equation and Diffusion equation}

We now send $\ve$ to $0$ and look at the density in the long time scale $t\ve^{-1}$:
\begin{equation}
\label{eq:timescaley}
{\tilde f}^\ve (q, \theta, t) = {f}^\ve (q, \theta, t \ve^{-1}). 
\end{equation}
Consider the particle density
$$ \tilde \rho^\ve_0 (q,t) =  \int_{-\pi}^{\pi} {\tilde f}^\ve (q, \theta, t)  d\theta.$$
When $\ve \to 0$, we have that $(\tilde \rho^\ve_0)_\ve$ converges to $\tilde \rho_0$ solution 
\begin{equation}
\label{eq:rho_0-equation02}
\partial_t \tilde \rho_0 + \langle V \rangle_G \cdot  \nabla_q \rho_0 =0.
\end{equation}
We can push forward the expansion and a fairly standard Chapman-Enskog expansion \cite{Chapman, ELM, LSR} (see Appendix \ref{app:chap}) gives the following approximated diffusion equation for the density:  
\begin{equation}
\label{eq:rho_0-equation2}
\partial_t \tilde \rho^\ve_0 + \langle V \rangle_G \cdot  \nabla \tilde \rho_0^\ve  - \ve \nabla\cdot  {\bf D} \, \nabla   \, \tilde \rho_0^\ve \; = O (\ve^2)
\end{equation}
where the symmetric matrix $\bf D$ of size $n$ is given by \eqref{eq:defD}.


\subsection{Finite size fluctuations and large deviations around the kinetic equation}
\subsubsection{Fluctuating kinetic equation}
When finite $N$ fluctuations are taken into account, beyond the `law of large number'~\eqref{eq:det-fluct-eq}, we obtain in the time scale $t\ve^{-1}$ (like in \eqref{eq:timescaley}) the very formal weak noise SPDE:
\begin{equation}
\begin{split}
\label{eq:f20}
\partial_t {\tilde f}^\ve &=\ve^{-1} \partial_\theta \left( \Gamma \left[ \partial_\theta U  + \cfrac{F( {\tilde f}^\ve) }{{\tilde \rho}^\ve}\right]   {\tilde f}^\ve  +  \Gamma \, \partial_\theta {\tilde f}^\ve \right) - V \cdot \nabla {\tilde f}^\ve \\
&+\sqrt{\frac{2}{N \ve} }\, \partial_\theta \Big(\sqrt{\Gamma{\tilde f}^\ve} \; \eta \Big).
\end{split}
\end{equation}
Here $\eta:=\eta (q, \theta,t)$ is a standard Gaussian noise $\delta$-correlated in $q$ and $\theta$, i.e. white in these variables. We rewrite the fluctuating kinetic equation as 
\begin{equation}
\label{eq:fluct_kin_start}
\partial_t {\tilde f}^\ve + \mc T ({\tilde f}^\ve) =\ve^{-1} \mc D_{{\tilde f}^\ve}({\tilde f}^\ve) +(\ve N)^{-1/2} \mc N\left(\sqrt{\Gamma {\tilde f}^\ve}\right)
\end{equation}
where $\mc N(g):=\sqrt{2}\, \partial_{\theta} (\eta \, g)$ is the noise operator. Recall \eqref{eq:rho_0-equation02} and \eqref{eq:rho_0-equation2}. It is then natural to look at the fluctuating kinetic equation at diffusive time scale in the frame defined by the transport equation \eqref{eq:rho_0-equation02}:
\begin{equation}
\label{eq:changetsfeps}
\begin{split}
{\bar f}^\ve (q, \theta,t) &:= \tilde f^\ve (q + \, t\ve^{-1} \langle V \rangle_G, \theta, t \ve^{-1}).
\end{split}
\end{equation}
which is solution of
\begin{equation}
\label{eq:fluct_kin_start045}
\begin{split}
&\ve \partial_t {\bar f}^\ve + {{\mc T}_0} ({\bar f}^\ve) =\ve^{-1} {\mc D}_{{\bar f}^\ve} ({\bar f}^\ve) + N^{-1/2} \mc N \Big(\sqrt{\Gamma {\bar f}^\ve}\Big),\\
\end{split}
\end{equation}
where the centered transport operator is defined for any function $g$ by
\begin{equation}
\label{eq:T000}
\begin{split}
{\mc T}_0 (g) := {\overline V} \cdot \nabla g\\
\end{split}
\end{equation}
with the vector field $\overline V$ defined by
\begin{equation}
\label{eq:u-def}
{\overline V} (\theta) = V(\theta) - \langle V \rangle_G.
\end{equation}

Equation \eqref{eq:f20}, \eqref{eq:fluct_kin_start}, \eqref{eq:fluct_kin_start045} are sometimes called ``Dean equation" \cite{Dean96} {\footnote{But it appeared previously in \cite{Dawson5} (see equation (0.8)).}}. For a formal derivation, see Appendix \ref{app:kin+fluc}.

\subsubsection{Fluctuating hydrodynamic equation}
It is tempting to extend the Chapman-Enskog expansion seen previously to pass from a kinetic equation to a hydrodynamic equation as $\ve \to 0$ in the context of the {\textit{fluctuating}} kinetic equation in order to get a {\textit{fluctuating}} hydrodynamic equation. This approach can be formally carried on, see Appendix \ref{app:chap+fluc}. However, at the difference of the (non fluctuating) Chapman-Enskog expansion which is in some cases under good mathematical control (see for instance \cite{LSR14} for a review on the fluid limits of the Boltzmann equation), there are serious difficulties with such approach when we take into account the finite size fluctuations. 

Indeed, the mathematical status of the Dean equation is dubious: even for finite $N$, it is difficult to make sense of the equation, from a rigorous point of view. By contrast, the large deviation principle that we develop in the next section has a clear meaning and is hence a safer starting point. Moreover, it provides interesting quantitative  informations about the macroscopic evolution of the system.  


\section{Main result: $\Gamma$- convergence of the rate function in the 
limit $\ve \to 0$ }
\label{sec:mainresult}


Before stating the main result of this paper we need to introduce a theoretical framework and some notation.

\subsection{Preliminary on $H^{-1}$ norms and $\Gamma$-convergence}

We first recall some basic facts about the notion of $\Gamma$-convergence and $H_{-1}$-norms. \\

The notion of $\Gamma$-convergence is a powerful notion to study limiting behavior of variational problems depending on some parameter, say $\nu$. If we aim to study the asymptotic behavior of $\inf_x F^\nu (x)$ as $\nu \to 0$, a natural but usually intractable strategy consists to compute a minimizer $x^\nu$ and to study the limit of $F^\nu (x^\nu)$. Instead, $\Gamma$-convergence avoids a direct computation of $x^\nu$ and provides a framework to approximate the family of variational problems $\inf_x F^\nu (x)$ by an effective variational problem $\inf_{x} F(x)$ where the functional $F$ is the ``$\Gamma$-limit" of the functionals $(F^\nu)_{\nu}$. In many cases, even if $\tilde F (x) =\lim_{\nu \to 0} F^\nu (x)$ exists for any $x$, the $\Gamma$-limit $F$ does not coincide with $\tilde F$, and while $\inf_x F^\nu (x)$ converges to $\inf_x F(x)$, it is not true that $\inf_x F(x) = \inf_x {\tilde F} (x)$. We refer the reader for example to \cite{Braides,DelMaso} for more informations and various examples. The connection between $\Gamma$-convergence and LDP problems is studied for example in \cite{Mauro1, DiM}.    

\begin{Definition}
\label{def:gamma-conv-def}
A sequence of functional $F^\nu: E \to \RR$ defined on some topological space $E$ $\Gamma$-converges to $F: E \to\RR$ as $\nu \to 0$ if 
\begin{enumerate}[1.]
\item for any $x \in E$ and any sequence $x^\nu \to x$, $\lim_{\mu \to 0}  \inf_{\nu \le \mu} F^\nu (x^\nu) \ge F (x)$ {\rm{($\Gamma$-liminf inequality)}};
\item there exists a sequence $x^\nu \to x$ such that $\lim_{\mu \to 0} \sup_{\nu \le \mu} F^\nu (x^\nu) \le F(x)$ {\rm{($\Gamma$-limsup inequality)}}.
\end{enumerate}
\end{Definition}

As we will see below the Large Deviations Functionals studied in this paper are expressed in terms of some weighted $H_{-1}$ norms.

\begin{Definition}
\label{def:H-1-norm}
Let $\Omega \subset \R^d$ be an open subset of $\R^d$ and $\chi : \Omega \to S_d^+ (\R)$ a function taking values in the set of positive definite symmetric matrices. The square of the $\chi$ weighted $H_{-1}$-norm of a scalar function $g :\Omega \to \R$ is defined by
\begin{equation}
 \label{eq:h-1-inf}
\left\Vert g \right\Vert_{-1, \chi}^2 = \inf_{c} \left\{ \int_\Omega \,  {c} \cdot \chi^{-1} {c}   \, d\omega\; ; \; \nabla \cdot {c} ={g}\right\}
\end{equation}
 where $\cdot $ is the usual scalar product on $\R^d$ and the infimum is carried over all smooth vector fields (called controls) ${c}:\Omega \to \R^d$. Alternatively it can be expressed by
 \begin{equation}
 \label{eq:h-1-sup}
\left\Vert g \right\Vert_{-1, \chi}^2 = 2 \sup_{\varphi} \left\{ \int_\Omega g \varphi d\omega -\tfrac{1}{2} \int \,  \chi \nabla \varphi\, \cdot \, \nabla \varphi  \, d\omega \right\}
\end{equation}
where the supremum is now taken over all smooth scalar functions $\varphi:\Omega \to \R$.
\end{Definition}

Since we want to study the $\Gamma$-limit of the rate functional \eqref{eq:LDF-ve} defined below  in terms of weighted $H_{-1}$-norms \eqref{eq:H-1-g-0}, the sup (resp. inf) representation will be useful to get the $\Gamma$-liminf (resp. the $\Gamma$-limsup).\\

\subsection{Kinetic large deviation functional}

We recall that we restrict our study to the case for which the set of local equilibria are all in the form $(q, \theta) \to \rho (q) V (\theta)$. 

The LDP with speed $N$ for the empirical density corresponding to the Dean equation \eqref{eq:fluct_kin_start045} on the time window $[0,T]$, was obtained by Dawson and G\"artner in the case $R=\infty$ \cite{Dawson,Dawson2}, and is given for any function $f:=f(q, \theta, t)$ by \cite{Tanaka,Dawson3,Kipnis2, Dawson5, Dawson4,Dai Pra, Fontbona,  Budhiraja, Fischer,Bar1,Bouchet, Muller, Reygner, Dos,Coppini}
\begin{equation}
\label{eq:LDF-ve}
\mathcal{I}_T^\ve(f) = \frac{1}{4} \int_0^T \left\lVert A_f^\ve (f) \right\rVert^2_{-1, \Gamma f} \, dt
\end{equation}
where
\begin{equation}
\label{eq:afeps}
\begin{split}
A_f^\ve (f)  &= \ve \partial_t f \, + \, \mc T_0  (f) \, -\, \ve^{-1} \, {\mc D}_f (f)\\
\end{split}
\end{equation}
with the $h>0$ weighted $H_{-1}$-norm {\footnote{To be precise, the norm defined is the standard quadratic norm in the $q$ variable and a weighted $H_{-1}$-norm in the $\theta$ variable.}} of the function $g:=g(q, \theta)$ defined by 
\begin{equation}
\label{eq:H-1-g-0}
\begin{split}
\lVert g \rVert^2_{-1,h} &= \inf_{\varphi} \left\{ \int \frac{\varphi^2}{h} \, dq\, d\theta~,~\partial_\theta \varphi = g \right\}\\
&= 2 \sup_{\varphi} \left\{ \int g \varphi \, dq \, d\theta\, -\,  \frac12 \int (\partial_\theta \varphi)^2 g \, dq\, d\theta \right\}. 
\end{split}
\end{equation}
In the formula above, the test functions $\varphi$ depend on position $q$ and angle $\theta$ and $h$ is evaluated at fixed time $t$.\\

\subsection{Linearized operator}

 We define the linear operator $\mathcal L_f$ as the linearized operator of the nonlinear operator ${\mathcal D}_f (f)$ at $f$, i.e.
\begin{equation}
\label{eq:lineariseL}
{\mathcal L}_f (g) := \lim_{\delta \to 0} \tfrac{{\mc D}_{f+\delta g} (f+\delta g) -{\mc D}_f (f) }{\delta}.
\end{equation}
In particular, if $f=f_{\rm{le}}$, we show in Appendix \ref{sec:linearizedoperator}  that ${\mathcal L}_{f_{\rm{le}}} ={\mathcal L}_G$ and that the latter acts on a test function $g$ as
\begin{equation}
\begin{split}
{\mathcal L}_G (g)&= \partial_\theta \left( \Gamma [ \partial_\theta U+ F(G) +  \partial_\theta] g \right) \\
&+  \partial_\theta \left( \Gamma G \Big[ F(g)  - \big\langle \tfrac{g}{G} \big\rangle_G \, F(G) \Big] \right). 
\end{split}
\label{LO}
\end{equation}
Note that thanks to \eqref{eq:fluxnul} we have that
\begin{equation}
\label{eq:LVV0}
{\mc L}_G (G) =0.
\end{equation}
Its adjoint with respect to the standard scalar product w.r.t. $d\theta$ is denoted by ${\mathcal L}_G^\dagger$ and its action on a test function $\varphi$ is given by
\begin{equation}
\label{eq:lvdagger}
\begin{split}
{\mathcal L}^\dagger_G (\varphi) &= - \Gamma [\partial_\theta U + F(G) ]\partial_\theta\varphi\, +\, \partial_\theta (\Gamma \partial_\theta \varphi) \\
&-  \left(F^\dagger ( G\Gamma  \partial_\theta \varphi)  - \big\langle  F(G) \Gamma  \partial_\theta \varphi \rangle_G \right)
\end{split}
\end{equation}
where $F^\dagger (\theta, \theta'):= F(\theta', \theta)$. Note that
\begin{equation}
\label{eq:LdaggerVV0}
{\mc L}_G^\dagger (1) =0.
\end{equation}
In the sequel we will assume the following {\textit{dissipative condition}}: 

\begin{equation}
\label{eq:vraiehypothese}
\begin{split}
&\text{If $g$ is such that $\int d\theta g =0$ then :} \int d\theta \, G^{-1} \,  {\mc L}_G (g) \, g \,  \le 0 \\
&\quad \text{with equality if an only if $g=0$}.\\
&\\
\end{split}
\end{equation}

\noindent As shown in the proof of Proposition \ref{prop:kernel-general} this condition implies that\\
\begin{equation}
\label{eq:kerLvdagger}
\begin{split}
&{\rm{Ker}} ({\mc L}_G^\dagger) ={\rm{Span}} (\bf 1), \quad {\rm{Ker}} ({\mc L}_G) ={\rm{Span}} (G).\\
&{\rm{Range}} ({\mc L}_G^\dagger) =\Big\{ u\; ;\; \langle u \rangle_G =0 \Big\}, \quad {\rm{Range}} ({\mc L}_G) =\Big\{ u\; ;\; \langle \tfrac{u}{G} \rangle_G =0 \Big\}.\\
\end{split}
\end{equation}

\begin{remark}
The equations \eqref{eq:LVV0} and \eqref{eq:LdaggerVV0} show that in the first line of \eqref{eq:kerLvdagger}, two inclusions always trivially hold. Moreover, it is easy to show that the first line of \eqref{eq:kerLvdagger} implies the second one since we recall that if $A$ is an operator then the range of $A^\dagger$  is equal to the orthogonal of the kernel of $A$.
\end{remark}

%
%


By \eqref{eq:kerLvdagger}, since $\overline V$ defined in \eqref{eq:u-def} is such that $\langle {\overline V}\rangle_G  =0$, there exists a vector field $\psi :=\psi(\theta) = (\psi_1 (\theta), \ldots, \psi_n (\theta)) \in \R^n$ such that
\begin{equation}
\label{eq:psi-def}
{\mc L}_G^\dagger \psi_k= - {\overline {V}}_k
\end{equation}
and the solution is unique up to some additive constant vector field.\\

We now introduce two square positive symmetric matrices of size $n$:  $\bf D$ (diffusivity) and $\BF \sigma$ (mobility). For any $\, k, \ell \in \{1,\ldots, n\}$ the entries of the matrices are defined {\footnote{For $\BF D$, the formulas give the same results if $\psi_i$ is replaced by $\psi_i +C_i$ where $C_i$ is an arbitrary constant.}} by:\\ 
\begin{equation}
\label{eq:defD}
\begin{split}
{\BF D}_{k \ell } &=  \frac{1}{2} \left(\left\langle \psi_k , {\overline V}_\ell  \right\rangle_G+ \left\langle \psi_\ell , {\overline V}_k  \right\rangle_G\right), 
\end{split}
\end{equation}
and
\begin{equation}
\label{eq:defsigma}
{\BF \sigma}_{k\ell}= \left\langle \partial_\theta \psi_k\,  ,\, \Gamma\,   \partial_\theta \psi_\ell  \right\rangle_G.
\end{equation}

\begin{remark}
The matrix $\BF \sigma$ is positive since if $x=(x_1,\ldots, x_n)\in \R^n$ then 
\begin{equation*}
x \cdot {\BF  \sigma }x = \left\langle \Gamma \,  [\partial_\theta \left( x \cdot  \psi \right)]^2 \right\rangle_G \ge 0
\end{equation*}
with equality if and only if for any $\theta$, $(x \cdot  \psi) (\theta)=0$ (we can always assume that $\psi$ is centered), which implies by \eqref{eq:psi-def} that $x \cdot \partial_\theta V (\theta) =0$. This cannot hold if $x\ne 0$ since we assumed \eqref{eq:nondegenerateV}. 
The fact that the matrice $\BF D$ is non-negative is a consequence of \eqref{eq:vraiehypothese} because
\begin{equation*}
x \cdot {\BF D} x = - \left\langle {\mc L}_G^\dagger \left( x \cdot \psi \right), x\cdot \psi  \right\rangle_G =- \int d\theta G^{-1} \; {\mc L}_G ( x \cdot G \psi) \; (x\cdot G \psi) \ge 0
\end{equation*}
with equality if and only if $x\cdot G\psi =0$, i.e. $x \cdot \psi =0$ (we can always assume that $G\psi$ is centered because $\psi$ can be chosen up to a constant vector field) which as above implies $x=0$.
\end{remark}


\begin{remark}
For the general model defined by \eqref{eq:general-model-manifold}, the only modifications are that $\partial_\theta$ has to be replaced by the gradient $\nabla_\theta$ and \eqref{eq:defsigma} has to be replaced by 
\begin{equation*}
{\BF \sigma}_{k\ell}= \sum_{a=1}^m \; \left\langle {\mathfrak g} \left(  A_a\, , \,  \nabla_\theta \psi_k\right)  \;  {\mathfrak g} \left( A_a \, , \,  \nabla_\theta \psi_\ell \right) \right\rangle_G.
\end{equation*}

\end{remark}

\subsection{Statement of the result}

We can now state our main result.

\begin{Theorem}
\label{thm:main}
Assume the dissipative condition  \eqref{eq:vraiehypothese} (and hence \eqref{eq:kerLvdagger}). Then, as $\ve \to 0$, the rate functional $\mathcal{I}_T^\ve$ in \eqref{eq:LDF-ve} $\Gamma$-converges to $\mathcal I_T$ given by 
\begin{equation}
\label{eq:LDF-IT}
\mathcal I_T (f) =
\begin{cases}
&\frac{1}{4} \int_0^T \, \left\Vert \partial_t \rho - \nabla \cdot {\bf D} \nabla \rho \right\Vert_{-1, \rho {\BF \sigma}}^2 \, dt \;\\
& \hspace{3cm} {\rm{if}} \; f(q, \theta,t) =f_{{\rm{le}}}(q, \theta,t)= \rho (q, t) G (\theta),\\
&+\infty \quad {\rm{otherwise}},
\end{cases}
\end{equation}
where the matrices $\bf D$ and $\BF \sigma$ are the matrices given by \eqref{eq:defD} and \eqref{eq:defsigma}. 
\end{Theorem} 

\begin{proof} 
This result is proved in Section \ref{sec:proofgamma}.
\end{proof}

We show in Proposition \ref{prop:kernel-general} that if the interaction is sufficiently weak then  \eqref{eq:vraiehypothese} holds. We observe however that the previous Theorem holds under weaker conditions, for example if \eqref{eq:kerLvdagger} is satisfied and if we have a unique solution to \eqref{eq:fluxnul}. Moreover, in \cite{BBCNP}, we focus on the active particle exemple \eqref{eq:pot-pc} where we will show that this proposition holds and that, by solving exactly \eqref{eq:psi-def} we can obtain explicit expressions for the diffusivity matrix \label{eq:defD} and for the mobility matrix \label{eq:defsigma}; we will then be able to infer some physical consequences for the physical system.

The form taken by the limiting functional is reminiscent of  the functional appearing in the Macroscopic Fluctuations Theory for diffusive systems \cite{Bertini} with the particular features that the diffusivity is independent of the density and the mobility is linear in the density. This is also the case for independent diffusion processes in the plane, but moreover there a proportionality between $\BF D$ and $\BF \sigma$ would hold and this is usually not the case here and in particular for the case \eqref{eq:pot-pc}. This absence of proportionality is a manifestation of the interactions at a macroscopic level. Observe that the limiting rate functional corresponds formally to Dean's equation for the empirical density evolving as
\begin{equation}
\label{eq:Dean-eq-rho}
\partial_{t} {\rho} =\nabla \cdot {\BF D}\nabla {\rho}  +\sqrt{\tfrac{2 }{N}}\nabla\cdot\left(\sqrt{{\rho} {\BF \sigma}}\, \xi\right)
\end{equation}
 with $\xi$ a standard $n$-space dimensional white noise.

\section{Proof of Theorem \ref{thm:main} }
\label{sec:proofgamma}

\subsection{Asymptotic expansion of $A^\ve$}

Recall the equation \eqref{eq:afeps}: 

\begin{equation}
\begin{split}
A_f^\ve (f)  &= \ve \partial_t f \, + \, \mc T_0  (f) \, -\, \ve^{-1} \, {\mc D}_f (f)\\
\end{split}
\end{equation}

Consider a sequence of densities approximating the local equilibrium at order $2$ in $\ve$:
\begin{equation}
\label{eq:feexp}
f^\ve = f_{\rm{le}}+ \varepsilon f_1 +\ve^2 f_2+ O(\ve^3). 
\end{equation}
We want to expand $A_{f^\ve}^\ve (f^{\ve})$ at first order in $\ve$. We have first (use \eqref{eq:lineariseL}) by a Taylor expansion that
\begin{equation*}
{\mc D}_{f^\ve} (f^\ve) ={\mc D}_{f_{\rm{le}}} (f_{\rm{le}}) + \ve {\mc L}_G (f_1) +\ve^2 \partial_\theta {\mc Q} (f_{\rm{le}}, f_1, f_2) + O(\ve^3). 
\end{equation*}
The presence of a $\partial_\theta$ in the last term follows from the presence of a $\partial_\theta$ on the left in the definition \eqref{eq:defDT} of ${\mc D}_f$. In the following we will not use the exact expression of $\mc Q$. Moreover we recall that ${\mc D}_{f_{\rm{le}}} (f_{\rm{le}}) =0$.
Then we get the Taylor expansion 
\begin{equation}
\label{eq:dla0-po}
\begin{split}
A_{f^\ve}^\ve (f^{\ve}) &= {\mc T}_0 (f_{\rm{le}}) - {\mc L}_G (f_1) \\
& +\, \ve \left[\partial_t f_{\rm{le}} +{\mc T}_0 (f_1) -\partial_\theta {\mc Q}  (f_{\rm{le}}, f_1, f_2)  \right] +O(\ve^2).
\end{split}
\end{equation}

\subsection{$\Gamma$-liminf}

In this subsection we prove the following proposition.
\begin{proposition}
\label{prop:Gammaliminf}
Assume \eqref{eq:vraiehypothese} (and hence \eqref{eq:kerLvdagger}).  Let $f$ be a time dependent angle-position density and consider a sequence  $(f^\ve)_{\ve}$ of time dependent angle-position densities in the form
\begin{equation}
f^\ve = f + \varepsilon f_1 +\ve^2 f_2+ O(\ve^3).
\label{seq} 
\end{equation}
Then we have the following $\Gamma$-liminf inequality
\begin{equation*}
\liminf_{\ve \to 0} \;  \mathcal{I}_T^\ve (f^\ve)  \; \ge \; \mathcal{I}_T (f).
\end{equation*}
where ${\mc I}_T$ is defined by \eqref{eq:LDF-IT}.
\end{proposition}

\begin{proof}
In order to simplify notation we denote by $\langle f, g \rangle$ the scalar product of the functions $f(q, \theta)$ and $g(q, \theta)$ with respect to $dq \, d\theta$. Recall the definition \eqref{eq:LDF-ve} of $\mc I_T^\ve$ and the Hamiltonian variational representation of the $H_{-1}$-norm in terms of a supremum given in \eqref{eq:h-1-sup}. For any test function $\varphi (q, \theta, t)$ we have 
\begin{equation}
\label{eq:lb-iteps}
{\mc I}_T^\ve (f^\ve) \ge \cfrac{1}{4} \int_0^T dt \left\{ 2 \langle \varphi, A_{f^\ve}^\ve (f^\ve)  \rangle  - \langle [\partial_\theta \varphi]^2, {\Gamma} f^\ve \rangle \right\}. 
\end{equation}
The aim is thus to choose a sequence $(\varphi^\ve)_{\ve}$ of test functions in order to maximize the righthand side of the previous expression in the limit $\ve \to 0$.

We observe first that if $f$ is not a local equilibrium  then ${\mathcal D}_f (f)\ne 0$ and ${\mathcal I}_T (f) =+\infty$. Hence if $(f^\ve)_{\ve}$ converges to $f$, the term $A_{f^\ve}^\ve (f^\ve)$ becomes equivalent as $\ve \to 0$ to $-\ve^{-1} {\mathcal D}_f (f)$ (see \eqref{eq:afeps}) so that by choosing the test function $\varphi=1$ in the previous formula we get
 $$\liminf_{\ve \to 0} {\mathcal I}_T^\ve (f^\ve) =+\infty.$$

Let now  $(f^\ve)_{\ve}$ be a sequence like in \eqref{eq:feexp} that converges towards a local equilibrium $f_{\rm{le}}(q, \theta, t) := \rho (q, t) V(\theta) $ as $\ve \to 0$. We consider a sequence of test functions $(\varphi^\ve)_{\ve}$ in the form 
\begin{equation*}
\varphi^\ve (q, \theta,t)= \ve^{-1}  \varphi_{-1} (q, \theta, t) +\varphi_0 (q, \theta, t).
\end{equation*}
By using \eqref{eq:dla0-po} and plugging this choice of test function in the righthand side of \eqref{eq:lb-iteps} we get
\begin{equation*}
{\mc I}_T^\ve (f^\ve) \ge \cfrac{1}{2} \int_0^T dt \left\{\ve^{-2} X_t + \ve^{-1} Y_t + Z_t \right\} \; +\; O(\ve) 
\end{equation*}
where \\
\begin{equation*}
\begin{split}
X &= -\tfrac{1}{2}\,  \langle (\partial_\theta \varphi_{-1})^2 , {\Gamma} f_{\rm{le}} \rangle\\
&\\
Y &= \left\langle \varphi_{-1}, {\mc T}_0 (f_{\rm{le}}) - {\mc L}_G (f_1)  \right\rangle -\tfrac{1}{2} \, \langle (\partial_\theta \varphi_{-1})^2 , \textcolor{black}{\Gamma}f_1 \rangle -\langle (\partial_\theta \varphi_{-1}) (\partial_\theta \varphi_{0}) , \textcolor{black}{\Gamma}f_{\rm{le}} \rangle  \\
&\\
Z &= \left\langle \varphi_{0}, {\mc T}_0 (f_{\rm{le}}) - {\mc L}_G (f_1)  \right\rangle + \left\langle \varphi_{-1} \, , \, \left[\partial_t f_{\rm{le}} +{\mc T}_0 (f_1) -\partial_\theta {\mc Q}  (f_{\rm{le}}, f_1, f_2)  \right] \right\rangle\\
&-\tfrac{1}{2} \, \langle (\partial_\theta \varphi_{-1})^2 ,\textcolor{black}{\Gamma} f_2 \rangle -\langle (\partial_\theta \varphi_{-1}) (\partial_\theta \varphi_{0}) , \textcolor{black}{\Gamma} f_{1} \rangle -\tfrac{1}{2} \, \langle (\partial_\theta \varphi_{0})^2 , \textcolor{black}{\Gamma}f_{\rm{le}}  \rangle 
\end{split}
\end{equation*}

Since $X$ is negative, in order to maximize the righthand side of \eqref{eq:lb-iteps} in the limit     $\ve \to 0$, we have to choose the test function in order to cancel $X$, i.e. 
\begin{equation*}
\varphi_{-1} (q, \theta,t) := \varphi_{-1} (q,t).
\end{equation*}
This implies that the second and third term in $Y$ are zero. Moreover the first one is also zero because:
\begin{itemize}
\item First, recalling the definition \eqref{eq:T000} of ${\mc T}_0$ and the definition \eqref{eq:u-def} of the vector field $\overline V$, we have that:
\begin{equation*}
\begin{split}
&\left\langle \varphi_{-1}, {\mc T}_0 (f_{\rm{le}}) \right\rangle = \left\langle  \varphi_{-1},  G {\mc T}_0 (\rho ) \right\rangle \\
&= \int dq \varphi_{-1} (q, t)  \left( \int d\theta G (\theta)  \, {\overline V} (\theta) \cdot \nabla \rho (q,t) \Big]\right)=0,
\end{split}
\end{equation*}
because $\langle {\overline V}\rangle_G =0$;\\
\item Secondly, using the expression \eqref{eq:lvdagger} of ${\mc L}_V^\dagger$ and the fact that  $\varphi_{-1}$ is independent of $\theta$, we have that:
\begin{equation*}
\left\langle \varphi_{-1}, {\mc L}_G (f_1)  \right\rangle = \left\langle {\mc L}_G^\dagger (\varphi_{-1}), f_1  \right\rangle=0.\\
\end{equation*}
\end{itemize}
For the term $Z$, since $\varphi_{-1}$ is independent of $\theta$, it can be simplified as 
\begin{equation}
\label{eq:ZZZ}
\begin{split}
Z &= \left\langle \varphi_{0}, {\mc T}_0 (f_{\rm{le}}) - {\mc L}_G (f_1)  \right\rangle + \left\langle \varphi_{-1} \, , \, \left[\partial_t f_{\rm{le}} +{\mc T}_0 (f_1)\right] \right\rangle  -\tfrac{1}{2} \, \langle (\partial_\theta \varphi_{0})^2 ,\textcolor{black}{\Gamma} f_{\rm{le}}  \rangle\\
&= \langle \varphi_{0}, {\mc T}_0 (f_{\rm{le}}) \rangle + \langle \varphi_{-1} \, , \, \partial_t f_{\rm{le}} \rangle  -\tfrac{1}{2} \, \langle (\partial_\theta \varphi_{0})^2 , \textcolor{black}{\Gamma}f_{\rm{le}}  \rangle\\
&- \langle f_1 \,,  \, {\mc L}_G^\dagger (\varphi_0) - {\mc T}_0^\dagger (\varphi_{-1}) \rangle.\\
&\\
\end{split}
\end{equation}
Observe that $Z$ does not depend on $f_2$. By definition of $\Gamma$-convergence the lower bound we have to obtain shall not depend on the way the sequence $(f^{\ve})_\ve$  converges to $f_{\rm{le}}$, i.e. shall not depend on $f_1$. A simple choice is then to take $\varphi_0$ solution of
\begin{equation}
\label{eq:psilop}
{\mc L}_G^\dagger \varphi_0 - {\mc T}_0^\dagger \varphi_{-1}=0 = {\mc L}_G^\dagger \varphi_0 + {\mc T}_0 \varphi_{-1}
\end{equation}
where the last equality is due to the property ${\mc T}_0^\dagger = -{\mc T}_0$. Since 
\begin{equation*}
{\mc T}_0 (\varphi_{-1}) (q, \theta, t)  = {\overline V} (\theta) \cdot \nabla \varphi_{-1} (q, t), 
\end{equation*}
by using the vector field $\psi$ defined in \eqref{eq:psi-def}, a solution to equation \eqref{eq:psilop} is provided by
\begin{equation*}
\varphi_0 (q, \theta, t)  = \psi(\theta) \cdot  \nabla \varphi_{-1} (q, t).
\end{equation*}
To summarize we get the following form for the test function
\begin{equation*}
\varphi^\ve (q, \theta,t) = \ve^{-1} \varphi_{-1} (q, t) + \psi(\theta) \cdot  \nabla \varphi_{-1} (q, t).  
\end{equation*}
We plug this form of the test function in \eqref{eq:ZZZ} and simplify each term. Recall the definition of ${\BF D}$ given in \eqref{eq:defD}. Using the definition \eqref{eq:T000} of ${\mc T}_0$ and performing one spatial integration by parts we have
\begin{equation*}
\langle \varphi_{0}, {\mc T}_0 (f_{\rm{le}}) \rangle = - \int dq \;  \varphi_{-1} (q, t)\;  \nabla \cdot {\bf D} \nabla \rho (q, t). 
\end{equation*}
For the second term we have trivially 
\begin{equation*}
\langle \varphi_{-1} \, , \, \partial_t f_{\rm{le}} \rangle = \int dq \; \varphi_{-1} (q, t) \; \partial_t \rho (q, t).
\end{equation*}
The third one is rewritten as
\begin{equation*}
\langle (\partial_\theta \varphi_{0})^2 ,\textcolor{black}{\Gamma} f_{\rm{le}}  \rangle = \int dq \;  \rho (q, t)\;  \nabla \varphi_{-1} (q, t)\; \cdot\,  {\BF \sigma} \nabla \varphi_{-1} (q,t) \; .
\end{equation*}
with the mobility matrix defined in \eqref{eq:defsigma}. Therefore we have obtained that
\begin{equation*}
\begin{split}
{\mc I}_T^\ve (f^\ve)&  \ge \cfrac{1}{4} \int_0^T dt \left\{ 2 \int dq \; \varphi_{-1} (q, t) \; \Big[ \partial_t \rho (q, t) \, dq -   \nabla \cdot {\bf D} \nabla \rho (q, t)\Big] \right.\\
&\left. \quad \quad \quad \quad  -   \int dq \;  \rho (q, t)\;  \nabla \varphi_{-1} (q, t)\;\cdot\,   {\BF \sigma} \nabla \varphi_{-1} (q,t) \right\} \; + \; O(\ve). 
\end{split}
\end{equation*}
Since $\varphi_{-1}$ is arbitrary we can take the supremum on $\varphi_{-1}$ on the righthand side of the previous expression and we get the result by recalling the variational formula in terms of a supremum for the $H_{-1}$-norm defining $\mc I_T$. 
 
\end{proof}

\begin{Remark}
Strictly speaking, we did not prove here in full rigor the $\Gamma$-liminf convergence of Definition \ref{def:gamma-conv-def} because we did not precise the topology setting and Proposition \ref{prop:Gammaliminf} is proved only for sequences in the form \eqref{seq}.
\end{Remark}

\subsection{$\Gamma$-limsup}

The aim of this section is to prove the following $\Gamma$-limsup property.

\begin{proposition}
Assume \eqref{eq:vraiehypothese} (and hence \eqref{eq:kerLvdagger}).  Let $f$ be a time dependent position-angle density function. There exists a sequence $(f^\ve)_\ve$ of time dependent position-angle density functions converging to $f$ such that
\begin{equation*}
\limsup_{\ve \to 0} {\mc I}_T^\ve (f^\ve) \le {\mc I}_T (f).
\end{equation*}
where ${\mc I}_T$ is defined by \eqref{eq:LDF-IT}.

\end{proposition}

\begin{proof}
We can assume that $f(q, \theta, t)$ is a local equilibrium in the form  $f_{\rm{le}} (q, \theta,t)=\rho(q,t)G(\theta)$ (otherwise the result is trivial to prove since the righthand side is then infinite). We now have to construct a sequence \textcolor{black}{(called a recovery sequence) } $(f^\ve)_\ve$ converging to $f_{\rm{le}}$ such that 
\begin{equation*}
\limsup_{\ve \to 0} {\mc I}_T^\ve(f^\ve) \; \le \; {\mc I}_T (f_{\rm{le}}).
\end{equation*}
Recall the formula \eqref{eq:LDF-ve} for ${\mc I}_T^\ve$ and the variational formula of the first line in \eqref{eq:H-1-g-0} for the $H_{-1}$-norm in terms of an infimum. Then we have that
\begin{equation*}
{\mc I}_T^\ve (f^\ve)  \le \cfrac{1}{4} \int_0^T dt \left\langle\frac{\big[ {c}^\ve \big]^2}{\textcolor{black}{\Gamma}f^\ve} \right\rangle
\end{equation*}
where $c^\ve:=c^\ve (q, \theta,t)$ is any control satisfying $\partial_\theta {c}^\ve = A_{f^\ve}^\ve (f^\ve)$. Observe that such a control exists only if 
\begin{equation}
\label{eq:constraintA}
\int_{-\pi}^\pi d\theta\,  A_{f^\ve}^\ve (q, \theta, t) =0
\end{equation}
for any $q$ and any time $t\le T$. Consider a sequence $(f^\ve)_{\ve}$ in the form
\begin{equation}
f^\ve = f_{\rm{le}}+ \varepsilon f_1.
\end{equation}
Since $f^\ve$ is a density this implies that $\langle 1, f_1 \rangle =0$. A Taylor expansion of $A_{f^\ve}^\ve (f^\ve)$ similar to the one given in \eqref{eq:dla0-po} shows that 
\begin{equation}
\begin{split}
A_{f^\ve}^\ve (f^{\ve}) &= {\mc T}_0 (f_{\rm{le}}) - {\mc L}_G (f_1) \\
& +\, \ve \left[\partial_t f_{\rm{le}} +{\mc T}_0 (f_1) -\partial_\theta {\mc Q}  (f_{\rm{le}}, f_1,0)  \right] +\ve^2 \partial_t f_1
\end{split}
\end{equation}
where the operator $\mc Q$ appeared in \eqref{eq:dla0-po} and whose exact form is irrelevant. Hence the constraint \eqref{eq:constraintA} is equivalent to 
\begin{equation*}
\begin{split}
&\int_{-\pi}^\pi d\theta \Big( {\mc T}_0 (f_{\rm{le}}) -{\mc L}_G (f_1)\Big) =0,\\
&\int_{-\pi}^\pi d\theta \Big(\partial_t f_{\rm{le}} +{\mc T}_0 (f_1)\Big) =0,\\
&\partial_t \left( \int_{-\pi}^\pi d\theta f_1 \right) =0.
\end{split}
\end{equation*}
The first constraint is always satisfied by recalling the definition \eqref{eq:T000} of ${\mc T}_0$ and observing that ${\mc L}_G^\dagger (1)=0$ (see \eqref{eq:lvdagger}). In the sequel we impose the following sufficient conditions
\begin{equation}
\label{eq:constraintsf1}
\begin{split}
&\int_{-\pi}^\pi d\theta \Big(\partial_t f_{\rm{le}} +{\mc T}_0 (f_1)\Big) =0,\quad  \int_{-\pi}^\pi d\theta f_1=0.
\end{split}
\end{equation}
Observe that the second condition implies $\langle 1, f_1 \rangle =0$. If such conditions hold then we have
\begin{equation}
\label{eq:optc1}
\limsup_{\ve \to 0} {\mc I}_T^\ve (f^\ve)  \le \cfrac{1}{4 } \inf_{c} \int_0^T dt \left\langle  \cfrac{c^2}{\textcolor{black}{\Gamma}f_{\rm{le}} } \right\rangle
\end{equation}
where the infimum is taken over all the controls $c$ such that 
\begin{equation}
\label{eq:amoniac}
\partial_\theta c ={\mc T}_0 (f_{\rm{le}}) -{\mc L}_G (f_1).
\end{equation}

Hence, the goal is now to choose $f_1$ respecting the constraints \eqref{eq:constraintsf1} and a corresponding control $c$ satisfying \eqref{eq:amoniac} in order to minimize the righthand side of the previous inequality. Given $f_1$, the control is unique up to a function depending only on position and time. Without the constraints the optimal control would be of course $c=0$, which would impose to $f_1$ to cancel the righthand side of \eqref{eq:amoniac}. 

We decompose then $f_1$ as the sum of two terms 
\begin{equation*}
f_1:= f_1^0 + g_1
\end{equation*}
where $f_1^0$ is such that
\begin{equation}
\label{eq:f10}
{\mc T}_0 (f_{\rm{le}}) -{\mc L}_G (f_1^0) =0.
\end{equation}
The naive choice $f_1=f_1^0$ would permit to take a zero control $c$ but the first constraint in \eqref{eq:constraintsf1} would not be respected. The term $g_1$ will permit to respect the constraint. 

We have that
\begin{equation*}
{\mc T}_0 (f_{\rm{le}}) (q, \theta, t)=  G(\theta) \, {\overline V} (\theta) \cdot \nabla  \rho (q, t).
\end{equation*}
Hence, we can solve \eqref{eq:f10} by writing 
\begin{equation}
\label{eq:f10-gr}
f_1^0 (q, \theta,t)  = - G (\theta) \omega (\theta) \cdot \nabla \rho (q, t)
\end{equation}
where the vector field $\omega:=\omega (\theta)=(\omega^1(\theta),\ldots, \omega^n (\theta)) \in \R^n$  is solution to
\begin{equation}
\label{eq:varphidef}
{\mc L}_G (G \omega_k) =- G {\overline V}_k,
\end{equation}
such that $\langle \omega \rangle_G =0$ (this is always possible since $\omega +C$ is also a solution for any constant vector field $C$).  The existence and uniqueness (up to additive constant vector fields ) of $\omega$  is a consequence of \eqref{eq:kerLvdagger}.  Observe now that by definition of ${\mc T}_0$ and of $f_1^0$, we have that
\begin{equation}
\label{eq:yash}
\begin{split}
\int_{-\pi}^\pi d\theta\,  {\mc T}_0 (f_1^0) &= - \int_{-\pi}^\pi d\theta \, {\overline V} (\theta) \cdot \nabla \Big( G(\theta) \omega(\theta) \cdot \nabla \rho (q,t)  \Big)  \\
&= - \nabla \cdot {\BF D} \nabla \rho
\end{split}
\end{equation}
where the last equality follows from the definition \eqref{eq:defD} of $\BF D$, the definition \textcolor{black}{\eqref{eq:psi-def}} of $\psi$ and the following computation
\textcolor{black}{
\begin{equation}
\label{eq:non-adjoint}
\begin{split}
\langle \psi_k , {\overline V}_\ell \rangle_G &=-   \int_{-\pi}^\pi d\theta  \psi_k (\theta)   [{\mc L}_G (G \omega_\ell)] (\theta) \\
&=   \int_{-\pi}^\pi d\theta  [{\mc L}_G^\dagger ( \psi_k)] (\theta) (G \omega_\ell) (\theta) = \langle {\overline V}_k, \omega_\ell \rangle_G.
\end{split}
\end{equation}}

\begin{remark}
Observe that if we had the relation 
\begin{equation*}
G^{-1}\circ\mathcal{L}_{G}\circ G = \mathcal{L}_{G}^{\dagger}
\end{equation*}
then \eqref{eq:non-adjoint} would be trivial to establish. However this last relation usuually does not hold.
\end{remark}
Hence, we can now reformulate the optimization problem \eqref{eq:optc1} as
\begin{equation}
\label{eq:optc2}
\limsup_{\ve \to 0} {\mc I}_T^\ve (f^\ve)  \le \cfrac{1}{4 } \inf_{c} \int_0^T dt \left\langle  \cfrac{c^2}{\textcolor{black}{\Gamma}f_{\rm{le}} } \right\rangle
\end{equation}
for any control $c$ such that 
\begin{equation}
\partial_\theta c = -{\mc L}_G (g_1).
\end{equation}
with the constraints \eqref{eq:constraintsf1} replaced by the following constraints on $g_1$
\begin{equation}
\label{eq:constraintsf2}
\begin{split}
&\partial_t \rho -  \nabla \cdot {\BF D} \nabla \rho  = -\int_{-\pi}^\pi d\theta {\mc T}_0 (g_1)\quad {\rm and} \quad \int_{-\pi}^\pi g_1 d\theta =0
\end{split}
\end{equation}
thanks to \eqref{eq:yash} and the fact that $\int_{-\pi}^\pi d\theta f_1^0  = 0$ (since $\langle \omega \rangle_G =0$).

We look now for a function $g_1$ in the form 
\begin{equation*}
g_1( q, \theta,t) = - a  (q, t) \cdot (G\xi)(\theta) 
\end{equation*}
where $\xi$ is a vector field of $\R^n$ such that $\langle \xi  \rangle_G =0$ (in order to respect the second constraint in \eqref{eq:constraintsf2}) and $a:=a(q,t) \in \R^n$ is an arbitrary vector field depending only on $q$ and $t$. We have then 
\begin{equation*}
-\int_{-\pi}^\pi d\theta {\mc T}_0 (g_1) = \nabla \cdot {\BF E} a
\end{equation*}
with $\BF E$ the \textcolor{black}{non symmetric} matrix defined by its entries as follows: 
\begin{equation}
\label{eq:matrixE}
{\BF E}_{k,\ell} =  \langle {\overline V}_k , \xi_\ell \rangle_G.
\end{equation}
We introduce the vector field  $W:=W(\theta)=(W_1 (\theta),\ldots, W_n (\theta)) \in \R^n$ such that 
\begin{equation}
\label{eq:defWfunc}
\partial_\theta W= {\mc L}_G (G \xi), \quad \int_{-\pi}^\pi d\theta \frac{W (\theta)}{\Gamma (\theta) G(\theta) }=0,
\end{equation}
and the square \textcolor{black}{symmetric} matrix $\BF R$ of size $n$ whose entries are defined by 
\begin{equation}
\label{eq:matrixR}
{\BF R}_{k\ell} = \left\langle \cfrac{W_k}{G} , \textcolor{black}{\frac{1}{\Gamma}} \cfrac{W_\ell}{G} \right\rangle_G, \quad k,\ell \in \{1,\ldots, n\}. 
\end{equation}
Observe then that 
\begin{equation*}
\begin{split}
-{\mc L}_G (g_1) (q, \theta, t)  &= a (q, t) \cdot  [{\mc L}_G (G\xi)] (\theta) \\
&= \partial_\theta \left[ a (q, t) \cdot  W (\theta)\right]
\end{split}
\end{equation*}
so that 
\begin{equation*}
c(q, \theta,t) := a (q,t)\cdot  W (\theta) 
\end{equation*}
is an admissible control in the optimisation problem \eqref{eq:optc2}. Observe moreover that
\begin{equation*}
\int_0^T dt \left\langle  \cfrac{c^2}{ \Gamma f_{\rm{le}} } \right\rangle= \int_0^T dt \int dq \cfrac{a (q,t) \cdot {\BF R} a (q,t) }{\rho (q, t)}.
\end{equation*}

We have obtained that for any $\xi (\theta)$ such that $\langle \xi \rangle_G=0$ 
(this choice for $\xi$ fix the matrix $\BF E$ and the matrix $\BF R$) and any $a:=a(q,t)$ satisfying the constraint 
\begin{equation}
\label{eq:constrainta}
\partial_t \rho -\nabla \cdot {\bf D} \nabla \rho = \nabla \cdot {\BF E} a
\end{equation}
there exists a sequence $(f^\ve)_\ve$ converging to $f_{\rm{le}}$ (depending on $f_1$, hence on $g_1$ and hence on $a, \xi$) such that 
\begin{equation}
\label{eq:opt3}
\limsup_{\ve \to 0} {\mc I}_T^\ve (f^\ve)  \le \cfrac{1}{4}  \int_0^T dt \int dq \cfrac{a (q,t) \cdot {\BF R} a (q,t) }{\rho (q, t)}.
\end{equation}
This recovery sequence is  given by

\begin{equation}
\label{eq:recoverysequence}
\begin{split}
f^\ve (q, \theta,t) &= \rho (q, t) G (\theta) - \ve G(\theta) \left[  \omega  (\theta) \cdot \nabla \rho (q,t) + \xi (\theta) \cdot a (q, t)\right].
\end{split}
\end{equation}

We want to make as small as possible the righthand side of \eqref{eq:opt3} by choosing $\xi$ and $a$. This optimal choice will then fix entirely the sequence $f^\ve$ defined in \eqref{eq:recoverysequence}. 

Given $\BF E$ and $\BF R$, recalling the definition \eqref{eq:h-1-inf} of $H_{-1}$-norm in terms of an infimum, we have that 
\begin{equation*}
\begin{split}
&\cfrac{1}{4} \inf_{a} \int_0^T dt \int dq \cfrac{a (q,t) \cdot {\BF R} a (q,t) }{\rho (q, t)}\\
&= \frac{1}{4 } \int_0^T \, \left\Vert \partial_t \rho - \nabla \cdot {\bf D} \nabla \rho \right\Vert_{-1, \rho {\BF{E R^{-1} E^\dagger}}}^2 \, dt
\end{split}
\end{equation*}
where the infimum above is taken of all controls $a$ satisfying \eqref{eq:constrainta}.

The challenge is then now to optimize over $\xi$ (the matrices ${\BF E}$ and ${\BF R}$ are functions of them) in order to make the righthand side of the previous equality as small as possible. By Proposition \ref{prop:Gammaliminf} we may guess that we have necessarily 
\begin{equation}
\label{eq:ineq-matrices-str}
\BF{E R^{-1} E^\dagger} \le {\BF \sigma}
\end{equation}
where the inequality is understood in terms of corresponding quadratic forms. This is indeed proved in Lemma \ref{lem:shur} below. In order to realize the equality we claim that it suffices to choose the vector field $\xi:=(\xi_1,\ldots, \xi_n)$ such that
\begin{equation}
\label{eq:choicechi}
\partial_\theta \psi_k = \cfrac{W_k}{\textcolor{black}{\Gamma}G}, \quad \text{i.e.} \quad {\mc L}_G (G\xi_k) =\partial_\theta \Big(\textcolor{black}{\Gamma} G \partial_\theta \psi_k \Big).
\end{equation}
The existence and uniqueness (because imposed to be centered) of $\xi$ is a consequence of \eqref{eq:kerLvdagger}. To show that with this choice we realize the equality in \eqref{eq:ineq-matrices-str}, we observe then first that by the definition of the mobility matrix \eqref{eq:defsigma} we get  
\begin{equation}
\label{eq:magie1}
{\BF \sigma}={\BF R}, 
\end{equation}
and secondly that 
\begin{equation}
\label{eq:magie2}
{\BF E}={\BF R}.
\end{equation}
The last equation come from
\begin{equation*}
\begin{split}
{\BF E}_{k \ell} &= \langle {\overline V}_k , \xi_\ell \rangle_G=-  \int_{-\pi}^\pi d\theta {\mc L}_G^\dagger (\psi_k) \, G \xi_\ell = - \int_{-\pi}^\pi d\theta \psi_k \, {\mc L}_G (G \xi_\ell) \\
&= \left\langle \partial_\theta \psi_k , \tfrac{W_\ell}{G} \right\rangle_G=  {\BF R}_{k \ell}
\end{split}
\end{equation*}
where the penultimate equality results from \eqref{eq:defWfunc} and   an integration by parts and the last one from \eqref{eq:choicechi}. In particular  \textcolor{black}{this} ${\BF E}$ is \textcolor{black}{finally} symmetric.
\textcolor{black}{Then, the two relations \eqref{eq:magie1} and \eqref{eq:magie2} give direcly the matricial equality  
\begin{equation*}
\BF{E R^{-1} E^\dagger} = {\BF \sigma}.
\end{equation*}
}


To summarize, with the choice of $\xi$ in \eqref{eq:choicechi} and the optimal control $a$ realizing the infimum in the righthand side of \eqref{eq:opt3} we have proved that the corresponding sequence $(f^\ve)_\ve$ defined by \eqref{eq:recoverysequence} satisfies
\begin{equation*}
\begin{split}
\limsup_{\ve \to 0} {\mc I}_T^\ve (f^\ve) \le  \frac{1}{4  }\int_0^T \, \left\Vert \partial_t \rho - \nabla \cdot {\bf D} \nabla \rho \right\Vert_{-1, \rho {\BF \sigma}}^2 \, dt.
\end{split}
\end{equation*}
Formally, the good choice of the recovery sequence is given by 
\begin{equation*}
\begin{split}
& f^{\varepsilon}(\boldsymbol{q},\theta,t)=\rho(\boldsymbol{q},t)G\left(\theta\right)\\
&+\varepsilon\left\{  \mathcal{L}_{G}^{-1}\left(G {\overline V} \right)\left(\theta\right) \cdot \nabla \rho(q,t) +\left(\mathcal{L}_{G}^{-1}\partial_{\theta} \Gamma G \partial_{\theta}\left(\mathcal{L}_{V}^{\dagger}\right)^{-1}\right)\left({\overline V} \right)\left(\theta\right)  \cdot a(q,t)  \right\}
\end{split}
\end{equation*}
with $a$ realizing the infimum in the righthand side of \eqref{eq:opt3}.
\end{proof}

\begin{lemma}
\label{lem:shur}
For any choice of the vector field $\xi$ satisfying $\langle \xi \rangle_G =0$ we have that
\begin{equation*}
\BF{E R^{-1} E^\dagger} \le {\BF \sigma}
\end{equation*}
where $\BF E$ and $\BF R$ are defined as functions of $\xi$ by \eqref{eq:matrixE} and \eqref{eq:matrixR}.
\end{lemma}

\begin{proof}
For any $\bx, \by \in \R^n$, recalling the definition \eqref{eq:defsigma} of $\BF \sigma$, we have by Cauchy-Schwarz inequality that 
\begin{equation*}
\begin{split}
\bx \cdot {\BF E} \by &= \left\langle \sum_k x_k \, \partial_\theta \psi_k \; ,\; \sum_\ell y_\ell \tfrac{W_\ell}{G} \right\rangle_G  \\
&\leq \sqrt{\left\langle \Big(\sum_k x_k \, \sqrt{\Gamma} \partial_\theta \psi_k\Big)^2  \right\rangle_{G}}    \sqrt{\left\langle \Big( \sum_\ell y_\ell \tfrac{1}{\sqrt{\Gamma}} \tfrac{W_\ell}{G} \Big)^2\right\rangle_{G}}\\
&=\sqrt{\bx \cdot {\BF \sigma} \bx}\sqrt{\by \cdot {\BF R} \by}
\end{split}
\end{equation*}
Observe now that ${\BF \sigma} - \BF{E R^{-1} E^\dagger}$ is the Schur complement of the block $\BF R$ of the symmetric matrix $\BF M$ defined by
\begin{equation*}
{\BF M} = \left[
\begin{array}{cc}
{\BF \sigma}&{\BF E}\\
{\BF E}^\dagger& {\BF R}
\end{array}
\right].
\end{equation*}
It is well known that if $M\ge 0$ then the Schur complement of the block $\BF R$ of the symmetric matrix $\BF M$ is also. So it is sufficient to prove that $\BF M$ is non-negative, which is a consequence of the inequality $\bx \cdot {\BF E} \by \le \sqrt{\bx \cdot {\BF \sigma} \bx}\sqrt{\by \cdot {\BF R} \by}$ proved above. 
\end{proof}

\section{Future work and open questions} 


\subsection{Homogenization limit first $\epsilon\rightarrow0$ first and then mean field limit $N\rightarrow\infty$ after} 

As mentioned in the introduction, we have the two different LDP principles:  \eqref{eq:LDP1-intro} obtained by fixing $N$ and sending $\varepsilon$ to $0$ and \eqref{eq:LDP2-intro} obtained by fixing $\ve$ and sending $N$ to $\infty$. In this work we studied the limit as $\ve \to 0$ of the rate functional appearing in \eqref{eq:LDP2-intro}. By a contraction principle we have therefore a LDP with a rate functional, say ${\mathcal F}_T (\rho)$ for the $q_i$'s density $\rho$ in the limit $N\to \infty$ and then $\ve \to 0$. From  \eqref{eq:LDP1-intro} we can deduce by a contraction principle a LDP for the empirical density of the $q_i$'s in the limit $\ve \to 0$ (with $N$ fixed). Then a natural question would be to study the limit as $N \to \infty$ of the corresponding rate function and understand the links the latter has with ${\mathcal F}_T$.\  Observe that related questions have been investigated in the finite dimensional case \cite{F1, Gartner2, Baldi, Majda,FS,Dupuis0} through the study of SDE's with a small noise regulated by a parameter $\alpha \to 0$ and fast oscillating coefficients whose oscillations are regulated by a second parameter $\delta \to 0$. The limiting behavior of the SDE depends on the relation between $\alpha$ and $\delta$.

\subsection{Non equilibrium models} 
We restricted  our study to the case where the local equilibria are unique and where the underlying $\theta_i$'s dynamics is reversible when $R=\infty$. None of theses conditions is necessary and probably that some of our results can be extended to cover situations where they do not hold. In particular it would be interesting to consider  `non-equilibrium' Shinomoto-Kuramoto type models for the velocity  \cite{Shinomoto,Ohta,Pikovsky,Giacomin} adapted  in our context, i.e. for example models with motion equations given by:
\begin{equation}
\label{eq: motion}
\begin{split}
& dq_i = \varepsilon \, V (\theta_i) dt,\\
&\dot{\theta}_i = F -h \sin \theta_i + \frac{\gamma}{\mathcal{N}_i}\sum_{j\in \mathcal{V}_i} \sin(\theta_j-\theta_i)  +\sqrt{2\Gamma}\eta_i(t).
\end{split}
\end{equation}
where $F$ is a constant force, hence not the derivative of a periodic force.

%

\section*{Acknowledgements}
This work has been supported by the projects EDNHS ANR-14- CE25-0011, LSD ANR-15-CE40-0020-01 of the French National Research Agency (ANR). This project has received funding from the European Research Council (ERC) under  the European Union's Horizon 2020 research and innovative programme (grant agreement No 715734).

\appendix

\section{Derivation of the kinetic equation}
\label{app:kin}
In this section we derive formally the kinetic equation \eqref{eq:det-fluct-eq}. Even if we are not very precise and careful in taking the different limits, we believe that the actual mathematical techniques should be sufficient to derive rigorously the previous kinetic equation (\cite{Bossy}). \\

Let us consider
$$f^N(q , \theta, t) \; dq d\theta:=f^{N,R,\varepsilon} (q , \theta, t) \; dq d\theta= \cfrac{1}{N} \sum_{i=1}^N \delta_{(q_i (t) ,\theta_i (t))} ( d q, d\theta)$$
the position-angle empirical density and let $\varphi (q, \theta)$ be a smooth compactly supported macroscopic observable. We have that
\begin{equation}
\label{eq:fluct1}
\begin{split}
&\cfrac{d}{dt} \, \int \varphi (q, \theta) \, f^N (q, \theta, t) \, d q d\theta \\
&= \varepsilon  \int  V(\theta) \cdot \nabla  \varphi (q, \theta) \, f^N (q, \theta, t) \, d q d\theta\\
& \, -\, \int (\partial_\theta U )(\theta)  \, \partial_\theta \varphi  (q, \theta) \, f^N (q, \theta, t) \, d q d\theta\\
&-\, \int \, \partial_\theta \varphi  (q, \theta) \,  g_{N,R, \varepsilon} (q, \theta, t) \, f^N (q, \theta, t) \, dq d \theta\\
& \; + \;  \cfrac{\sqrt 2}{N} \, \sum_{i=1}^N  \Gamma (\theta_i)\, \partial_\theta \varphi (q_i, \theta_i) \, \cfrac{{dW}_i}{dt}  \; +\; \, \int    \Gamma (\theta)\,  \partial_\theta^2 \varphi (q, \theta) \, f^N (q, \theta, t) \, d q d\theta 
\end{split}
\end{equation}
with
\begin{equation*}
g_{N,R,\ve} (q, \theta,t)  = \cfrac{\int dq' d\theta' \, {\bf 1}_{| q -q'| \le R} F(\theta,\theta') f^N (q', \theta',t) }{\int dq' d\theta' \,{\bf 1}_{| q -q'| \le R}  f^N (q', \theta',t) }.
\end{equation*}
The last term on the RHS of \eqref{eq:fluct1} is due to the It\^o correction. Basic stochastic calculus shows that
$$\EE \left[ \left( \cfrac{1}{N} \, \sum_{i=1}^N \int_0^t \Gamma (\theta_i) \, \partial_\theta \varphi (q_i, \theta_i) \, dW_i (s) \right)^2  \right] =O(N^{-1})$$
vanishes in the large $N$ limit. Assuming now that as $N \to \infty$, $f^N$ converges to some function $f^{R, \varepsilon}$ we get that 
\begin{equation*}
g_{N,R, \varepsilon} (q, \theta, t) \rightarrow {\bar g}_{R, \varepsilon} (q, \theta,t) :=\cfrac{\int dq' d\theta' \, {\bf 1}_{| q -q'| \le R} F(\theta,\theta') f^{R,\ve} (q', \theta',t) }{\int dq' d\theta' \,{\bf 1}_{| q -q'| \le R}  f^{R,\ve} (q', \theta',t) }.
\end{equation*}
Observe now that as $R \to 0$, assuming that $\lim_{R\to 0} f^{R, \varepsilon}=f^\varepsilon$, we have that 
\begin{align*}
&\lim_{R \to 0} {\bar g}_{R, \varepsilon} (q, \theta, t) =\cfrac{F(f^\ve)}{\Pi (f^\ve)}.
\end{align*}
By performing some integration by parts, we conclude that, in distribution, 
\begin{equation*}
\lim_{R\to 0}\lim_{N \to \infty} f^{N,R, \varepsilon} =f^\varepsilon
\end{equation*}
where $f^{\varepsilon}$ is the deterministic solution of the kinetic equation \eqref{eq:det-fluct-eq}.

 \section{Local equilibria}
\label{sec:app-le} 


We look for the solutions $f:=f(q, \theta)$ of ${\mc D}_f (f) =0$. In view of \eqref{eq:defDT} there exists then a function $C(q)$ of $q$ only such that
\begin{equation*}
\Gamma \left[ \partial_\theta U + \,  \cfrac{ F (f) }{\Pi (f) }\right]  f  +  \Gamma \, \partial_\theta f  =C.
\end{equation*}
Dividing by $\Gamma f$ on both sides, we remark that the lefthand side is a derivative in $\theta$ because $\Pi (f)$ is independent of $f$ and $F(f)=\partial_\theta W(f)$. Hence the integral in $\theta$ of the lefthand side divided by $f$  is zero which implies that $C(q)=0$. Moreover, if $f (q, \theta)$ is a solution it is necessarily in the form $f(q, \theta) = \rho (q) G_{\rho (q)} (\theta)$ with $\int d\theta G_\rho =1$ and $\rho (q) = \int_{-\pi}^\pi d\theta f(q, \theta)$. We then observe that $G_\rho (\theta)$ will be solution of the equation with unknown $G$ 
\begin{equation}
\label{eq:equationforV}
\left[ \partial_\theta U +\,  F (G) \right]  G +  \partial_\theta G  =0.
\end{equation}
We assume there exists a single (normalized) solution to this equation that we denote by $G$. This corresponds to an absence of phase transition.  Then all local equilibria are in the form $\rho (q) G(\theta)$. Observe that \eqref{eq:equationforV} can be rewritten as a fixed point problem
\begin{equation}
\label{eq:equationforV2}
G=T(G):={e^{-H}}, \quad \partial_\theta H=\partial_\theta U+ F(G) \quad \text{and} \quad {\int d\theta e^{-H}}=1.    
\end{equation}
The map $T$ is a contraction mapping for the uniform topology if the interaction coupling $F$ is sufficiently small and then in this case the uniqueness of $G$ follows.

\section{Linearized operator}
\label{sec:linearizedoperator}

\subsection{Expression of the linearized operator ${\mc L}_f$} 

For a given position-angle density $f:=f(q, \theta)$ we compute the linearized operator ${\mc L}_f$ of ${\mc D}_f (f)$ as defined by \eqref{eq:lineariseL}. We perform hence a first order Taylor expansion in $\delta$ for ${\mc D}_{f+\delta g} (f+\delta g)$ in the direction given by the function $g$ (which satisfies $\langle 1, g \rangle =0$). We have
\begin{equation*}
\begin{split}
&{\mc D}_{f+\delta g} (f+\delta g) \\
&= \partial_\theta \left\{ \Gamma \left[ \partial_\theta U +  \frac{F(f)+\delta F(g)}{\Pi (f) +\delta \Pi (g)}\right] (f+\delta g) +\Gamma [\partial_\theta f + \delta \partial_\theta g] \right\}. 
\end{split} 
\end{equation*}
Since 
\begin{equation*}
\frac{F(f)+\delta F(g)}{\Pi (f) +\delta \Pi (g)} = \frac{F(f)}{\Pi (f)}  +\delta \left( \frac{F(g)}{\Pi (f)} - \frac{F(f)\Pi(g)}{(\Pi (f))^{2}}\right) \, +\, O(\delta^2),
\end{equation*}
we get 
\begin{equation*}
\begin{split}
&{\mc D}_{f+\delta g} (f+\delta g) = {\mc D}_{f} (f) + \delta {\mc L}_f (g)  +O(\delta^2)
\end{split} 
\end{equation*}
with
\begin{equation}
\label{eq:expressionL_f}
{\mc L}_f (g) := \partial_\theta \left( \Gamma \left[ \partial_\theta U + \frac{F(f)}{\Pi(f)}\right] g  +\Gamma \partial_\theta g +\Gamma \left( \frac{F(g)}{\Pi(f) } - \frac{F(f)}{(\Pi(f))^2} \Pi(g) \right) f \right). 
\end{equation}

In particular, if $f (q, \theta)=f_{\rm{le}} (q, \theta) = \rho(q) G(\theta)$ is a local equilibrium, since $F(f_{\rm{le}}) =\rho F (G)$ and $\Pi(f_{\rm{le}}) =\rho)$ we have that ${\mc L}_{f_{\rm{le}}}$ is independent of $\rho$, i.e. 
\begin{equation*}
{\mc L}_{f_{\rm{le}}} ={\mc L}_G
\end{equation*}
and is given by  \eqref{LO}.

\subsection{Properties of ${\mc L}_G$ and $\mathcal{L}_{G}^{\dagger}$  }

In this section we give sufficient conditions for the validity of the assumption \eqref{eq:vraiehypothese} (and hence  \eqref{eq:kerLvdagger}). Roughly speaking, we prove that if the interaction coupling function is sufficiently small then  \eqref{eq:vraiehypothese} is satisfied.\\

We start by proving a lemma giving some bound for $G$. We use the notation $\| \cdot\|_{\infty}$ to denote the supremum norm of bounded functions.

\begin{lemma}
\label{lem:C1-the}
There exist universal constants $K^*$ and $K_*$ such that
\begin{equation}
\label{eq:boundG}
\sup_{\theta} G(\theta) \le K^* \exp\left\{ 2\pi  [C + \|\partial_\theta \log \Gamma\|_\infty ]  \right\}
\end{equation}
and
\begin{equation}
\label{eq:boundG2}
\inf_{\theta} G(\theta) \ge  K_* \exp\left\{ 2\pi  [C + \|\partial_\theta \log \Gamma\|_\infty ]  \right\}
\end{equation}
where 
\begin{equation*}
C:=C(F,U)=  \| \partial_\theta U\|_\infty + \| F \|_\infty .
\end{equation*}
\end{lemma}

\begin{proof}
Notice first that by \eqref{eq:equationforV2}, we have that $G \ge  0$. Moreover since $\int_{-\pi}^\pi d\theta G(\theta) =1$ this implies there exists some $\theta^*$ such that $G(\theta^*) \le 1/\pi$ (the constant is not optimal). Moreover for all $\theta$ we have that
\begin{equation*}
| [F(G) ](\theta)| = \left| \int d\theta' F(\theta, \theta') G (\theta') \right| \le \| F\|_\infty 
\end{equation*}
since $\int d\theta' G(\theta') =1$. By dividing \eqref{eq:equationforV} by $G$, observing that $\partial_\theta G /G = \partial_\theta (\log G)$, and integrating between $\theta^*$ and $\theta$ we deduce that
\begin{equation*}
G(\theta) \le G(\theta^*) \exp\left\{   | \theta -\theta^*|[C + \|\partial_\theta \log \Gamma\|_\infty ] \right\}
\end{equation*}
which gives \eqref{eq:boundG} thanks to the choice of $\theta^*$.

To get \eqref{eq:boundG2} we proceed similarly by reasoning with $1/G$ instead of $G$. There exists $\theta_*$ such that $G(\theta_*) \ge 1/4\pi$ since $G\ge  0$ and $\int d\theta G(\theta) =1$. By dividing \eqref{eq:equationforV} by $G$, observing that $\partial_\theta G /G = - \partial_\theta (\log (1/G))$, and integrating between $\theta_*$ and $\theta$ we deduce that
\begin{equation*}
(1/G) (\theta) \le (1/G)(\theta_*) \exp\left\{   | \theta -\theta_*|[C + \|\partial_\theta \log \Gamma\|_\infty ] \right\}
\end{equation*}
which gives \eqref{eq:boundG2} thanks to the choice of $\theta_*$.
\end{proof}

\begin{proposition}
\label{prop:kernel-general}
There exists a constant $K$ depending on  $\bb U$ and $\Gamma$ such that if 
\begin{equation*}
\| \partial_\theta W\|_{\infty} \le K
\end{equation*}
then the following holds: there exists a constant $\kappa>0$ such that for any differentiable function $g$ such that $\int d\theta g(\theta) =0$  we have
\begin{equation}
\label{eq:neg:linearized-op}
- \int d\theta \, G^{-1} \,  {\mc L}_G (g) \, g \,  \ge  \kappa \int d\theta \, G\, \Gamma\,  \big[ \partial_\theta (g/G)\big]^2,
\end{equation}
and consequently, we have that
\begin{equation}
\label{eq:consequencelemme}
{\rm{Ker}} ({\mc L}_G) = {\rm{Span}} (G), \quad {\rm{Ker}} ({\mc L}^\dagger_G) = {\rm{Span}} ({\bf 1}).
\end{equation}

\end{proposition}

\begin{proof}
We first prove \eqref{eq:neg:linearized-op}. Recall the potential $H$ introduced in \eqref{eq:equationforV2} satisfying $G=e^{-H}$. By \eqref{LO}, for any smooth function $g$ such that $\int d\theta g (\theta)=0$,  we have that
\begin{equation}
\begin{split}
{\mathcal L}_G (g)&= \partial_\theta \left( \Gamma e^{-H} \partial_\theta (e^H g)\right)  +  \partial_\theta \left( \Gamma G F(g)\right). 
\end{split}
\end{equation}
Multiplying this expression by $e^H g$, integrating in $\theta$ and performing an integration by parts we get
\begin{equation}
\label{eq:funaki1}
-\int d\theta \, e^H\,  {\mathcal L}_G (g)\, g = {\bb D} (g)  \; +\; \int d\theta \, \Gamma e^{-H} F(g) \partial_\theta (e^H g)
\end{equation}
where
\begin{equation*}
 {\bb D} (g):=\int d\theta \, e^{-H}\, \Gamma\,  \big[ \partial_\theta (e^H g)\big]^2 \ge 0.
\end{equation*}
By Cauchy-Schwarz inequality the second term on the right hand side of \eqref{eq:funaki1} can be bounded as
\begin{equation*}
\begin{split}
\left\vert \int d\theta \, \Gamma e^{-H} F(g) \partial_\theta (e^H g) \right\vert &\le \sqrt{{\bb D}(g)}\;  \sqrt{ \int d\theta \, \Gamma e^{-H} F^2 (g)}
\end{split}
\end{equation*}
and the goal is thus now to prove that 
\begin{equation}
\label{eq:funaki2}
\int d\theta \, \Gamma e^{-H} F^2 (g) \le \kappa {\bb D} (g) 
\end{equation}
for a constant $\kappa <1$ independent of $g$. By Cauchy-Schwarz inequality we have that
\begin{equation*}
\begin{split}
\| F(g)\|_{\infty}^2 &\le \| F\|_{\infty}^2 \left( \int d\theta  |g(\theta)| \right)^2\\
& \le  \| F\|_{\infty}^2 \, \left( \int d\theta \, e^{-2 H (\theta)} \right) \, \left( \int d\theta  |e^H g(\theta)|^2\right)\\
& \le  2\pi \| F\|_{\infty}^2 \| e^{-H}\|_{\infty}^2 \, \left( \int d\theta  |e^H g(\theta)|^2\right).
\end{split} 
\end{equation*}
By Poincar\'e inequality we have that
\begin{equation*}
\int d\theta  |e^H g(\theta)|^2 \le \int d\theta \, \big[ \partial_\theta (e^H g)\big]^2 \le \| 1/\Gamma\|_{\infty} \| e^H \|_{\infty} \, {\bb D} (g). 
\end{equation*}
Recalling that $G= e^{-H}$ we get that \eqref{eq:funaki2} is satisfied with 
\begin{equation*}
\kappa := 2\pi  \| F\|_{\infty}^2 \, \| \Gamma\|_\infty \| 1/\Gamma\|_{\infty} \|1/G \|_{\infty} \, \|G\|_{\infty}^2
\end{equation*}
Thanks to Lemma \ref{lem:C1-the} we see that if $ \| F\|_{\infty}$ is sufficiently small, $\kappa<1$ and this concludes the proof of the main result of the proposition.\\

To deduce \eqref{eq:consequencelemme}, let $g \in {\rm {Ker}}({\mc L}_G)$ so that ${\mc L}_G (g)=0$. We consider $h:= g- c G$ with $c=\int d\theta \, g$ so that $\int d\theta\, h =0$. Since ${\mc L}_G (G) =0$, we have also ${\mc L}_G (h)=0$. Then, since we have 
$$0 =\int d\theta \, G^{-1} \, {\mc L}_G (h) \, h   \le\, -\kappa \int d\theta \, G\, \Gamma\,  \big[ \partial_\theta (h/G)\big]^2$$
we deduce that $h/G$ is constant and since its integral in $\theta$ of $G$ is $1$ while the integral of $h$ in $\theta$ is $0$, we deduce that $h=0$, i.e. $g=c G \in {\rm{Span}} (G)$. Similarly if $\varphi \in {\rm{Ker}} ({\mc L}_G^\dagger)$, we start to define ${\widehat \varphi} = \varphi-c$ where the constant $c$ is such that $\int d\theta  G \widehat \varphi =0$, i.e. $c=\langle \varphi \rangle_G$. Since ${\mc L}_G^\dagger ({\bf 1})=0$, we have ${\mc L}_G^\dagger (\widehat \varphi)=0$. We use \eqref{eq:neg:linearized-op} to write
\begin{equation*}
0= \int d\theta\, G\,  {\mc L}_G^\dagger ({\widehat \varphi}) \; {\widehat \varphi} =  \int d\theta\, G^{-1}\,  {\mc L}_G (G{\widehat \varphi}) \; (G {\widehat \varphi}) \le - \kappa \int d\theta \, G \Gamma\,  [\partial_\theta {\widehat \varphi}]^2.
\end{equation*}
It follows that $\widehat \varphi$ is constant and since it is of mean zero, it is zero. Hence $\varphi$ is constant, i.e. $\varphi \in {\rm{Span}} ({\bf 1})$.

\end{proof}

\section{Chapman-Enskog expansion in the homogenized limit $\ve \to 0$}


We define $\Pi_G$ the projection on the vector space of local equilibria given for any function $g$ by 
\[
[\Pi_G( g )] (q, \theta) := \left(\int_{-\pi}^{\pi} g (q, \theta)  d\theta\right) \, G (\theta).
\]

\subsection{Chapman-Enskog expansion of the kinetic equation}
\label{app:chap}

We now look at the density in the long time scale $t\ve^{-1}$:
\begin{equation*}
{\tilde f}^\ve (q, \theta, t) = {f}^\ve (q, \theta, t \ve^{-1})
\end{equation*}
and we then send $\ve$ to $0$. By \eqref{eq:det-fluct-eq} we have that
\begin{equation}
\label{eq:f200}
\partial_t {\tilde f}^\ve + \mc T ( {\tilde f}^\ve)=\ve^{-1} \mc D_{{\tilde f}^\ve} ({\tilde f}^\ve). 
\end{equation}
Equation \eqref{eq:f200} will be the basis of the following expansion. 

Let $\tilde f_0^\ve$ be a local equilibrium defined by ${\Pi_G} (\tilde f^\ve) = \tilde f_0^\ve$, i.e.
\begin{equation*}
{\tilde f}_0^\ve (q, \theta, t) = {\tilde \rho}_0^\ve (q, t) G(\theta), \quad \tilde \rho_0^\ve:= \Pi (\tilde f_0^\ve),
\end{equation*}
and let us define $g_1^\ve$, assumed to be of order $1$ as $\ve \to 0$, by:
$$\tilde f^\ve = \tilde f_0^\ve+ \ve g_1^\ve.$$
In other words, the hydrodynamic behavior of $\tilde f^\ve$ is entirely captured by $\tilde f_0^\ve$. Observe that ${\Pi_G} (g_1^\ve)=0$ by construction. Inserting this expansion into \eqref{eq:f200}, we obtain:
\begin{equation}
[\partial_t + \mc T]\, (\tilde f_0^\ve) = \mc L_{{\tilde f}_0^\ve} (g^\ve_1) +O(\ve).
\label{eq:ep0}
\end{equation}
Notice that $\mc L_{\tilde f_0^\ve}={\mc L}_G$ defined in \eqref{LO} because $\tilde f_0^\ve$ is a local equilibrium. Applying ${\Pi_G}$ to \eqref{eq:ep0} yields 
\begin{equation}
\label{eq:jb007}
\Big[ {\Pi_G}\,  (\partial_t + \mc T) \Big] \, (\tilde f_0^\ve) =O(\ve), 
\end{equation}
because ${\Pi_G} \, \mc L_G =0$ and this implies
\begin{equation}
\begin{split}
&\partial_t \tilde \rho^\ve_0 + \langle V   \rangle_G \cdot  \nabla  \tilde \rho^\ve_0 = O(\ve)
\end{split}
\end{equation}
where we recall that $\langle \cdot \rangle_G$ denotes the expectation w.r.t. $G$. The last equation is the hydrodynamical equation at leading order when $\ve \to 0$.\\

Our goal is now to compute the  $O(\ve)$ correction term. Observe that the equation:
\[
\mc L_G (\psi) = \big[\partial_t + \mc T \big] ( \tilde f_0^\ve)
\]
in general has no solution for $\psi$, because ${\Pi_G}$ applied on the right hand side does not exactly vanish while ${\Pi_G} \, \mc L_G =0$. However, using \eqref{eq:jb007} we can as well rewrite \eqref{eq:ep0} as
\[
\mc L_{G} (g_1^\ve) = \big[ {\rm Id}-{\Pi_G} \big] \big[\partial_t + \mc T\big] \, ( \tilde f_0^\ve) +O(\ve).
\]
Removing the $O(\ve)$, the equation
\[
\mc L_{G} (\psi) = \big[ {\rm Id}-{\Pi_G} \big]  \big[ \partial_t + \mc T\big] (\tilde f_0^\ve)
\]
together with the condition that ${\Pi_V} (\psi)=0$, has a unique solution denoted by $\tilde f_1^\ve$ thanks to assumption \eqref{eq:kerLvdagger}. We have
\begin{align*}
&\left( \big[ {\rm Id}-{\Pi_V} \big]\big[\partial_t + \mc T\big] \, (\tilde f_0^\ve)\right)  \,  (q, \theta,t) = {\overline V} (\theta) \cdot \nabla \tilde \rho^\ve_0 (q, t) .
\end{align*}
Let  $\omega$ be the vector field solution to
\begin{equation}
\label{eq:omegaomegadeuxfois}
{\mc L}_G (V \omega ) =- G {\overline V},
\end{equation}
such that $\langle \omega \rangle_V =0$ (this is always possible since $\omega +C$ is also a solution for any constant vector field  $C$).  The existence and uniqueness of $\omega$ is a consequence of \eqref{eq:kerLvdagger} (see also \eqref{eq:varphidef} where this vector field is introduced to prove the $\Gamma$-limsup). Therefore we have
\begin{equation*}
\tilde f_1^\ve (q, \theta, t) = - \nabla\tilde \rho^\ve_0 (q,t)  \cdot (G\omega ) (\theta).
\end{equation*}
We then rewrite $g_1^\ve = \tilde f_1^\ve + \ve g_2^\ve$ and so defined $g_2^\ve$ will be of order $1$. We have 
$$\tilde f^\ve = \tilde f_0^\ve + \ve \tilde f_1^\ve + \ve^2 g_2^\ve.$$
Plugging this in \eqref{eq:f200} we get
\begin{equation*}
\begin{split}
& \big[\partial_t + \mc T\big] (\tilde f_0^\ve) + \ve \, \big[\partial_t + \mc T\big] (\tilde f_1^\ve) + \ve^2\,  \big[ \partial_t + \mc T \big] (g_2^\ve) =  \ve^{-1} \mc D_{\tilde f^\ve} ({\tilde f}^\ve).
 \end{split}
\end{equation*}
We apply ${\Pi_G}$ on both sides and observe that ${\Pi_G} {\mc D}_{\tilde f^\ve} ({\tilde f}^\ve) =0$, $[{\Pi_G} \, \partial_t ](\tilde f_1^\ve)=[\partial_t \,  {\Pi_G}] (\tilde f_1^\ve) =0$ since ${\Pi_G} (\tilde f_1^\ve) =0$. It follows that 
\begin{equation}
\label{eq:ep1}
 {\Pi_G} [\partial_t + \mc T]\,  (\tilde f_0^\ve) + \ve \, {\Pi_G} \mc T \, (\tilde f_1^\ve)  =O(\ve^2).
\end{equation}
Observe now that $\tilde f_1^\ve$ has the same expression as $f_1^0$ in \eqref{eq:f10-gr} (by changing there $\rho$ by ${\tilde \rho}_0^\ve$). Therefore by using the same computations as in \eqref{eq:yash} and \eqref{eq:non-adjoint} 
\begin{equation*}
{\Pi_G} \mc T\,  (\tilde f_1^\ve) = -\ve \nabla \cdot {\bf D} \nabla  \tilde \rho_0^\ve.
\end{equation*}
Then we obtain the following approximated diffusion equation for the density  
\begin{equation*}
\partial_t \tilde \rho^\ve_0 + \langle V \rangle_G \cdot \nabla \tilde \rho^\ve_0 - \ve \nabla \cdot {\bf D} \nabla  \tilde \rho_0^\ve \; = O (\ve^2)
\end{equation*}
where the matrix $\bf D$ is given by \eqref{eq:defD}.

\subsection{Formal derivation of the fluctuating kinetic equation}
\label{app:kin+fluc}

Since we are interested in the large fluctuations around the limiting typical behavior described in Section \ref{sec:kl}, we have to take in account the first order corrections (in $N$), i.e. to remember that we neglected the small noise term in \eqref{eq:fluct1}
\begin{equation}
\label{eq:intermezzo}
\cfrac{\sqrt 2}{N} \, \sum_{i=1}^N \Gamma (\theta_i) \partial_\theta \varphi (q_i, \theta_i) \, {\dot \eta}_i (t)=\sqrt{\cfrac{2}{N}} \, \cfrac{1}{\sqrt N} \, \sum_{i=1}^N \Gamma (\theta_i) \partial_\theta \varphi (q_i, \theta_i) \, {\dot \eta}_i (t)
\end{equation}
which, in the large $N$ limit and then small $R$ limit, may be approximated by 
\begin{equation*}
\sqrt{\frac{2}{N}}\partial_\theta \; \Big(\sqrt{ \Gamma f^\ve}\;  \eta \Big)
\end{equation*}
where $\eta:=\eta (q,\theta,t)$ is a $(q,\theta,t)$-Gaussian white noise. Observe that this is a non-trivial assumption since first the previous term is mathematically difficult to define and secondly because this results from the belief that the correlations in the sum \eqref{eq:intermezzo} may be neglected. Therefore, in order to take into account fluctuations, we have to replace \eqref{eq:det-fluct-eq} by the fluctuating kinetic equation 
\begin{equation}
\begin{split}
\label{eq:f2}
\partial_t f^\ve &=\partial_\theta \left( \Gamma \left[ \partial_\theta U  + \cfrac{ F (f^\ve) }{\rho^\ve}\right]   f^\ve  +  \Gamma \, \partial_\theta f^\ve \right) -\varepsilon V(\theta) \cdot  \nabla f^\ve \\
&\quad +\sqrt{\frac{2}{N}}\, \partial_\theta \big(\sqrt{\Gamma f^\ve} \; \eta \big).
\end{split}
\end{equation}

We now send $\ve$ to $0$ and look at the fluctuating density in the long time scale $t\ve^{-1}$:
\begin{equation*}
{\tilde f}^\ve (q, \theta, t) = {f}^\ve (q, \theta, t \ve^{-1}). 
\end{equation*}
We have that
\begin{equation}
\begin{split}
\label{eq:f2}
\partial_t {\tilde f}^\ve &=\ve^{-1}\partial_\theta \left( \Gamma \left[ \partial_\theta U  + \cfrac{ F (f^\ve) }{\rho^\ve}\right]   f^\ve  +  \Gamma \, \partial_\theta f^\ve \right) - V(\theta) \cdot \nabla {\tilde f}^\ve \\
&\quad +\sqrt{\frac{2 }{N \ve} }\, \partial_\theta \big(\sqrt{\Gamma {\tilde f}^\ve} \; \eta \big).
\end{split}
\end{equation}
Performing a change of frame and accelerating again time by $\varepsilon^{-1}$ like in \eqref{eq:changetsfeps}, we obtain \eqref{eq:fluct_kin_start045}.


\subsection{Chapman-Enskog expansion of the fluctuating kinetic equation}
\label{app:chap+fluc}

We would like to proceed as in the previous section, with a Chapman-Enskog expansion. There are now two small parameters, $\ve$ and $N^{-1}$, and we will have to choose an appropriate scaling.
We introduce explicitly $N$ in the notation.
The local equilibrium $\tilde f_0^{\ve,N}$ is defined by ${\Pi_V} (\tilde f^{\ve,N}) = \tilde f_0^{\ve,N}$ (hence ${\mc L}_{\tilde f_0^{\ve,N}} ={\mc L}_G$) and the correction $g_1^{\ve,N}$ by:
$$\tilde f^{\ve,N} = \tilde f_0^{\ve,N}+ \ve g_1^{\ve,N}.$$
It is not clear a priori that $g_1^{\ve,N}$ can be taken of order $1$; we assume however that  $\ve g_1^{\ve,N} =o(1)$. Inserting this into \eqref{eq:fluct_kin_start}, we obtain
\begin{equation}
\label{eq:ordre0}
[\partial_t +\mc T] \, (\tilde f_0^{\ve,N}) =\mc L_{G} (g_1^{\ve,N}) +(\ve N)^{-1/2} \mc N \Big(\sqrt{\Gamma {\tilde f}^{\ve,N}}\Big) +o(1) 
\end{equation}
Notice we have not expanded the noise term. Applying ${\Pi_G}$ to the above equation yields
\[
{\Pi_G} [\partial_t +\mc T] \, (\tilde f_0^{\ve,N}) = o(1)
\]
which provides the hydrodynamic equation at leading order; it is not modified by the noise.
We now rewrite \eqref{eq:ordre0} as
\begin{equation*}
\begin{split}
\mc L_{V} (g_1^{\ve,N}) &= \big( {\rm Id}-{\Pi_G} \big) (\partial_t + \mc T)\, ( \tilde f_0^{\ve,N} ) + (\ve N)^{-1/2} \mc N \Big(\sqrt{\Gamma {\tilde f_0}^{\ve,N}}\Big)\\
& +o(1) + (\ve N)^{-1/2} O(\ve g_1^{\ve,N})
\end{split}
\end{equation*}
where we have now expanded the noise: this creates a noisy term of order $\ve g_1^{N,\ve}$, denoted by $O (\ve g_1^{\ve,N})$. 
We call $\tilde f_1^{\ve,N}$ the unique solution of
\[
\mc L_{G} (\tilde f_1^{\ve,N}) = \big( {\rm Id}-{\Pi_G} \big) (\partial_t + \mc T)\, ( \tilde f_0^{\ve,N})  + (\ve N)^{-1/2}  \mc N \Big(\sqrt{\Gamma {\tilde f_0}^{\ve,N}} \Big), \quad {\Pi_G}(u)=0.
\]
Since
\begin{equation*}
\begin{split}
 &\Big[ \big( {\rm Id}-{\Pi_G} \big) (\partial_t + \mc T) (\tilde f_0^\ve) \Big] (q, \theta,t) + (\ve N)^{-1/2} \mc N\Big(\sqrt{\Gamma {\tilde f_0}^{\ve,N}}\Big)\\
 &=  \nabla {\tilde \rho}_0^\ve (\bq,t) \cdot {\bar V} (\theta) \,   G (\theta) +  (\ve N)^{-1/2}\sqrt{2 \tilde \rho_0^\ve (q, t)}  \,  \partial_\theta \left(\sqrt{(\Gamma G) (\theta)} \eta (q,\theta,t) \right),
\end{split}
\end{equation*}
we get that (recall \eqref{eq:omegaomegadeuxfois})
\begin{equation}
\label{eq:f1epsNio}
\begin{split}
\tilde f_1^{\ve,N}&= -   \nabla {\tilde \rho}_0^\ve (\bq,t) \cdot (V\omega)  (\theta)\; +\;  (\ve N)^{-1/2} \sqrt{2 \tilde \rho_0^\ve (q, t)} \; \nu (\theta, q,t),
\end{split}
\end{equation}
with 
\begin{equation}
\label{eq:nu}
{\mc L}_G (\nu) (q, \theta, t) = \partial_\theta \left(\sqrt{(\Gamma G) (\theta)} \, \eta (q,\theta,t) \right).
\end{equation}
We can always choose $\nu$ such that $\EE (\nu) =0$ since ${\rm{Ker}} ({\mathcal L}_G)={\rm{Span}} (G)$.  Formally, the contribution to $\tilde f_1^{\ve,N}$ given by the first term in the right hand side is of order $1$, and the contribution of the noise, second term in the right hand side is of order $(\ve N)^{-1/2}$.
We rewrite $g_1^{\ve,N} =\tilde f_1^{\ve,N} +\ve g_2^{\ve,N}$, where we want that $\ve g_2^{\ve,N} =o \Big(   \tilde f_1^{\ve,N}$ \Big). The full expansion is then
\begin{equation*}
    \begin{split}
        &\tilde f^{\ve,N}= \tilde f_0^{\ve,N} +\ve \tilde f_1^{\ve,N} +\ve^2 g_2^{\ve,N}.
    \end{split}
\end{equation*}
At this point we can make sure that the expansion makes sense, that is 
$\ve \tilde f_1^{\ve,N}=o(1)$. Formally, this requires only that $N^{-1/2} =o(1)$ i.e. $N$ large. However, if we want that $\ve \tilde f_1^{\ve,N}$ is actually $O(\ve)$, we have to require that
$(\ve N)^{-1} =O(1)$.
We plug again the expansion for $\tilde f^{\ve,N}$ into \eqref{eq:fluct_kin_start}:
\begin{equation*}
\begin{split}
&[\partial_t +\mc T] (\tilde f_0^{\ve,N}) +\ve [\partial_t +\mc T] ( \tilde f_1^{\ve,N}) +O(\ve ^2 g_2^{\ve,N})\\
&\quad \quad = \ve^{-1} \mc D_{\tilde f^{\ve,N}} (\tilde f^{\ve,N}) +(\ve N)^{-1/2} \mc N(\sqrt{\Gamma \tilde{f}^{\ve,N}}), 
\end{split}
\end{equation*}
and we apply ${\Pi_G}$. The right hand side vanishes, and we are left with
\begin{equation}
{\Pi_G} [\partial_t +\mc T] \,  (\tilde f_0^{\ve,N}) +\ve {\Pi_G} \mc T \,  (\tilde f_1^{\ve,N})  = O(\ve ^2 g_2^{\ve,N}).
\end{equation}
We assume that the right hand side is indeed much smaller than the second term in the left hand side in the scaling limit. Recall \eqref{eq:f1epsNio}. We observe now that by using the same computations as in \eqref{eq:yash} and \eqref{eq:non-adjoint} 
\begin{equation*}
[{\Pi_G} \mc T]\,  (\nabla {\tilde \rho}_0^\ve \cdot (V\omega)) = - G\, \nabla \cdot {\bf D} \nabla  \tilde \rho_0^\ve.
\end{equation*}
with $\BF D$ defined by \eqref{eq:defD}, and we claim that 
\begin{equation*}
\begin{split}
[{\Pi_G} \mc T] \,  \Big(\sqrt{\tilde \rho_0^\ve} \; \nu\Big) (q, \theta,t) =G(\theta)  \nabla\cdot\left(\sqrt{\tilde\rho_0^\ve (q,t) \BF{\sigma}}\,  \zeta (q,t) \right)
\end{split}
\end{equation*}
with $\zeta:=\zeta(q,t)$ a standard $2$-space dimensional Gaussian white noise and $\BF \sigma$ defined by  \eqref{eq:defsigma}. To prove this write
\begin{equation*}
\begin{split}
&\left[[{\Pi_G} \mc T] \,  \Big(\sqrt{\tilde \rho_0^\ve} \; \nu\Big) \right]\, (q, \theta, t) = G(\theta) \, Z^\ve (q,t),\\
&\quad Z^\ve (q,t) := \int d\theta' \, {\overline V} (\theta') \cdot \nabla \Big(\sqrt{\tilde \rho_0^\ve} \; \nu \Big) (q, \theta', t) 
\end{split}
\end{equation*}
where $Z^\ve (q,t)$ is a centered random variable. Recalling  \eqref{eq:nu} and \eqref{eq:psi-def} we have that
\begin{equation*}
\begin{split}
Z^\ve (q,t) &= - \int d\theta \, ({\mathcal L}_G^\dagger \psi)  (\theta) \cdot \nabla \Big(\sqrt{\tilde \rho_0^\ve} \; \nu \Big) (q, \theta, t) \\
&= - \int d\theta \, \psi (\theta) \cdot \nabla\Big(\sqrt{\tilde \rho_0^\ve} \;  {\mathcal L}_G (\nu) \Big) (q, \theta, t)\\
&=- \int d\theta \, \psi (\theta) \cdot \nabla\Big(\sqrt{\tilde \rho_0^\ve} \;  \partial_\theta (\sqrt{\Gamma G} \eta )  \Big) (q, \theta, t)\\
&= \int d\theta \, {\sqrt{\Gamma G}} (\theta) (\partial_\theta\psi) (\theta) \cdot \nabla\Big(\sqrt{\tilde \rho_0^\ve} \; \eta  \Big) (q, \theta, t)\\
&= \nabla \cdot \left[  \int d\theta \, {\sqrt{\Gamma G}} (\theta) \Big(\sqrt{\tilde \rho_0^\ve} \; \eta  \Big) (q, \theta, t)  \; (\partial_\theta\psi) (\theta) \right]:=\nabla \cdot Y^\ve (q, t) 
\end{split}
\end{equation*}
with $Y^\ve$ a centered Gaussian field whose covariance satisfies
\begin{equation*}
\begin{split}
\EE (Y^\ve (q,t) Y^{\ve} (q', t') ) &= \delta (q-q') \delta (t-t') \, {\tilde \rho}_0^\ve (q, t) \, \langle \Gamma (\partial_\theta \psi) \cdot (\partial_\theta \psi) \rangle_G.
\end{split}
\end{equation*}

This provides the fluctuating hydrodynamic equation we are looking for. The final stochastic PDE for $\tilde \rho_0^\ve$ is given by (compare with \eqref{eq:Dean-eq-rho}):
\begin{equation}
\begin{split}
&\partial_t \tilde \rho^\ve_0 + \langle V \rangle_G \cdot \nabla \tilde \rho_0^\ve =  \varepsilon \nabla \cdot \mathbf{D} \nabla \tilde\rho_0^\ve +\sqrt{\frac{2\ve }{N}}\nabla\cdot\left(\sqrt{\tilde\rho_0^\ve \BF{\sigma}} \zeta \right) \; + o(1).
\label{eq:rhofinal2}
\end{split}
\end{equation}

\end{document}